\documentclass[11pt]{article}   	
\usepackage{graphicx}	
\usepackage{amsmath}
\usepackage{amssymb}
\usepackage{amsthm}
\usepackage{color}			
\usepackage{booktabs}
\usepackage{multirow}
\usepackage{amssymb}
\usepackage{algorithmicx}
\usepackage[ruled]{algorithm}
\usepackage[title,titletoc]{appendix}
\usepackage{algpseudocode}
\usepackage{rotating}
\usepackage{hyperref,cite}
\usepackage{caption}
\usepackage{subfig}
\usepackage{accents}
\usepackage{tcolorbox}
\usepackage[bottom]{footmisc}
\usepackage{authblk}
\usepackage{longtable}

\numberwithin{equation}{section}
\newtheorem{theorem}{Theorem}

\newtheorem{proposition}{Proposition}
\newtheorem{remark}{Remark}
\newtheorem{lemma}{Lemma}

\usepackage[title,titletoc]{appendix}
\usepackage{algpseudocode}
\graphicspath{{figures/}}

\title{Alternating minimization for square root principal component pursuit}

\author{Shengxiang Deng\thanks{School of Data Science, Fudan University, Shanghai, China \texttt{(sxdeng21@m.fudan.edu.cn)}.}, Xudong Li \thanks{School of Data Science, Fudan University, Shanghai, China \texttt{(lixudong@fudan.edu.cn)}.}, Yangjing Zhang \thanks{Institute of Applied Mathematics, Academy of Mathematics and Systems Science, Chinese Academy of Sciences, Beijing, China \texttt{(yangjing.zhang@amss.ac.cn)}.}}

\begin{document}
	\maketitle
	\begin{abstract}
		Recently, the square root principal component pursuit (SRPCP) model has garnered significant research interest. It is shown in the  literature that the SRPCP model guarantees robust matrix recovery with a universal, constant penalty parameter. While its statistical advantages are well-documented, the computational aspects from an optimization perspective remain largely unexplored.
		In this paper, we focus on developing efficient optimization algorithms for solving the SRPCP problem. Specifically, we propose a tuning-free alternating minimization (AltMin) algorithm, where each iteration involves subproblems enjoying closed-form optimal solutions.
		Additionally, we introduce techniques based on the variational formulation of the nuclear norm and Burer-Monteiro decomposition to further accelerate the AltMin method. Extensive numerical experiments confirm the efficiency and robustness of our algorithms.
	\end{abstract}

	
	\maketitle

	\section{Introduction}
	Low-dimensional models are ubiquitous in modern high-dimensional data analysis. Various algorithmic and theoretical advances in retrieving low-dimensional information such as sparsity, and low-rankness from  massive
	high-dimensional datasets are made in recent years. These developments have undoubtedly revolutionized the practice of data science and signal processing, and found applications in many areas, such as image processing, bioinformatics, and machine learning. In this paper, we focus on recovering a low-rank matrix from corrupted observations.
	Specifically, we study the following observation model:
	\begin{equation}\label{obser-model}
		D = L_0 + S_0 + Z_0,
	\end{equation}
	where $D\in\mathbb{R}^{n_1\times n_2}$, without loss of generality, $n_1\geq n_2$, is a given observation matrix,  $L_0$ is an unknown low-rank matrix, $S_0$ contains sparse corruptions, and $Z_0$ is some dense noise. We aim to accurately and effectively estimate $L_0$ and $S_0$.

	This model has been extensively studied in the literature, see  \cite{candes2011robust,chandrasekaran2011rank}. In the seminar work of \cite{candes2011robust}, %
	model \eqref{obser-model} was studies under the noiseless setting, i.e., $Z_0$ was assumed to be absent in model \eqref{obser-model}.
	Let $\|X\|_{*} $ and $\|X\|_1 = \sum_{ij} |X_{ij}|$ denote the nuclear norm and the $\ell_1$-norm
	of the matrix $X$, respectively.  \cite{candes2011robust} %
	showed that under some assumptions on $L_0$ and $S_0$ to avoid the  identifiability issues, it is possible to exactly recover $(L_0, S_0)$ by solving the following convex principal component pursuit problem:
	\begin{equation}
		\label{prob:pcp}
		\min \left\{ \|L\|_* + \lambda\|S\|_1 \mid L + S = D
		\right\},
	\end{equation}
	where $\lambda = 1/\sqrt{n_1}$ is the universal regularization parameter, and is obviously independent of the rank of $L_0$ and the sparsity of $S_0$. Later,   \cite{zhou2010stable} %
	studied the more practical noisy setting where $Z_0$ in model \eqref{obser-model} is not vacuous. Assuming similar separation and incoherence conditions as in \cite{candes2011robust} and upper boundedness of the Frobenius norm of $Z_0$, i.e., $\|Z_0\|_F\le \delta$ for some known $\delta > 0$, they established recovery results for the optimal solution pair to the following stabilized generalization of \eqref{prob:pcp}:
	\begin{equation}
		\label{prob:spcp}
		\min \left\{ \|L\|_* + \lambda \|S\|_1 \mid \|L + S - D\|_F \le \delta
		\right\}
	\end{equation}
	with $\lambda = 1/\sqrt{n_1}$.
	Instead of solving  \eqref{prob:spcp} directly,
	\cite{zhou2010stable} proposed in the simulation study to solve the following reformulation
	\begin{equation}
		\label{prob:uspcp}
		\min_{L,S}\,  \|L\|_* + \lambda \|S\|_1 + \frac{\rho}{2} \|L + S - D\|_F^2,
	\end{equation}
	where $\lambda = 1/\sqrt{n_1}$, $\rho >0$ is heuristically chosen based on the prior knowledge of $Z_0$. It is not difficult to observe that problems \eqref{prob:spcp} and \eqref{prob:uspcp} are equivalent in the sense that
	there exists a correspondence between $\rho>0$ and $\delta >0$ such that the optimal solution sets of \eqref{prob:spcp} and \eqref{prob:uspcp} are the same. Comparing to \eqref{prob:spcp}, problem \eqref{prob:uspcp} is considered to be more approachable from the optimization algorithmic aspect. Indeed, problem \eqref{prob:uspcp} can be readily solved via various first-order algorithms, such as the famous accelerated proximal gradient method \cite{nesterov27method}. Theoretical recovery results based on model \eqref{prob:uspcp} (or similar ones) are further investigated and strengthened by %
	\cite{klopp2017robust,wong2017matrix,chen2021bridging}. A main criticism of models \eqref{prob:spcp} and \eqref{prob:uspcp} is that their constructions rely on the prior knowledge of the noise level to determine the parameter $\rho$.  However, this information is usually not available ahead of time. Thus, in practice, to obtain meaningful estimators, the computationally expensive step of tuning these parameters is unavoidable.
	
	To overcome this difficulty, %
	\cite{zhang2021square} proposed a new framework termed square root principal component pursuit. They showed that under standard identifiability assumptions as in \cite{candes2011robust,zhou2010stable}, any optimal solution pair to
	\begin{equation}
		\label{prob:srpcp_zhang}
		\min_{L,S}\,\,
		\|L\|_* + \frac{1}{\sqrt{n_1}} \|S\|_1 + \sqrt{\frac{n_2}{2}}  \|L + S - D\|_F
	\end{equation}
	recovers $(L_0, S_0)$ with a recovery error of order ${\cal O}(\sqrt{n_1n_2}\|Z_0\|_F)$. We note that different from the squared penalty term in \eqref{prob:uspcp}, the square root penalty term is used in \eqref{prob:srpcp_zhang}, which makes the universal choice of $\rho =  \sqrt{n_2/2}$ possible. Thus,  \eqref{prob:srpcp_zhang} is a tuning-free model for recovering  $(L_0, S_0)$. Similar ideas of using square-root loss to obtain tuning-free methods for various high-dimensional regression applications can be found in \cite{belloni2011square,sun2012scaled,stucky2017sharp,yang2022optimal}. In \cite{zhang2021square}, slack variables $(L_1, S_1, Z)$ were further introduced to equivalently reformulate model \eqref{prob:srpcp_zhang} into
	\begin{equation}
		\label{prob:srpcp_ADMM}
		\begin{array}{cl}
			\min  & \|L_1\|_*  + \frac{1}{\sqrt{n_1}} \|S_1\|_1 + \sqrt{\frac{n_2}{2}} \|Z\|_F \\[3mm]
			\mbox{s.t.}& L + S + Z = D, \, L - L_1 = 0, \, S - S_1 = 0.
		\end{array}
	\end{equation}
	Then, the two-block alternating direction method of multipliers (ADMM) by \cite{glowinski1975approximation,gabay1976dual} was used to solve the above reformulated model. While the main focus of \cite{zhang2021square} is on the statistical recoverability of model \eqref{prob:srpcp_zhang}, emphasizing  the square-root loss and the universal choice of $\rho$, the practical efficiency, as well as theoretical convergence properties, of the used ADMM algorithm were not fully addressed.
	
	In this paper, we propose to further examine the model
	\begin{equation}\label{prob:srpcp}
		\min_{L,\,S}\,\,\|L\|_* + \lambda\|S\|_1 + \mu\|L+S-D\|_F
	\end{equation}
	with an emphasis from the optimization perspective. First, we reinterpret  problem \eqref{prob:srpcp} as a distributionally robust optimization (DRO) problem, which aims to find solutions that minimize the worst-case squared loss over some specially designed uncertainty set.
	As a compliment to the current literature, this new interpretation deepens our understandings of model \eqref{prob:srpcp},
	and allows us to tune the parameter $\mu$ to obtain meaningful estimators even without the widely used separation and incoherence assumptions.
	In fact, the relationships between regularized square-root loss problems and robust optimization have been observed widely in the literature; see, e.g., \cite{el1997robust,xu2010robust,xu2009robustness,shafieezadeh2015distributionally,bertsimas2018characterization}.
	Secondly, we propose to  solve problem \eqref{prob:srpcp} directly by the Alternating Minimization (AltMin) method \cite{tseng2001convergence,beck2015convergence}. The AltMin method or its general form, namely the block coordinate descent method (see, e.g., \cite{tseng2001convergence}), is a fundamental and popular method for solving optimization problems of multiple blocks of decision variables. At each iteration, the objective function is minimized with respect to one of the blocks while the remaining is fixed. %
	Specifically, the AltMin  method  for solving \eqref{prob:srpcp} minimizes the two variables $S$ and $L$ in turn and takes the following updating rule at the  $i$-th iteration:
	\begin{equation}
		\label{alg:alm}
		\left\{
		\begin{aligned}
			& S^{i+1} \in\arg\min_S \,\, \{  \lambda\|S\|_1 + \mu\|L^{i}+S-D\|_F \}, \\[2pt]
			& L^{i+1} \in\arg\min_L \,\, \{ \|L\|_* +  \mu\|L+S^{i+1}-D\|_F \}.
		\end{aligned}	
		\right.
	\end{equation}
	The two involved subproblems in \eqref{alg:alm} are shown to have closed-form optimal solutions (see Sections~\ref{subsec:updateS} and \ref{subsec:updateL}). This nice observation not only makes the AltMin method particularly suitable and easy to implement for solving  \eqref{prob:srpcp}, but it may also be of independent interest. Compared to the ADMM method, the above AltMin algorithm has much simpler updating rules and is tuning-free. We also conducted a comprehensive analysis on the convergence properties of the AltMin method.
	Thirdly, to further enhance the numerical performance of our alternating minimization approach, we propose to accelerate the AltMin method by reformulating the ``$L$''-subproblem  using the variational formulation of the nuclear norm and the Burer-Monteiro (BM) decomposition technique.  This idea has been widely used in the optimization community to improve the practical performance of various algorithms, see, e.g., \cite{lee2022accelerating,rennie2005fast}. In our case, this approach allows us to use the cheap partial Singular Value Decomposition (SVD) to alleviate the heavy computational burden of the full SVD in the updating of $L$ in \eqref{alg:alm}. Although the new subproblem is nonconvex,	
	careful examinations in Section~\ref{sec:acc} reveal that it still enjoys closed-form and easy-to-compute optimal solutions.
	Another advantage of our acceleration strategy is that we are able to provide a simple, theoretically guaranteed strategy for estimating the rank that needs to be computed in the BM-decomposition.
	The effectiveness of the proposed acceleration strategy is demonstrated through extensive numerical experiments.
	
	The remaining parts of this paper are organized as follows. In the next section, we review
	some key concepts and settings in the distributionally robust optimization and
	establish the equivalent DRO representation of model \eqref{prob:srpcp}. In Section~\ref{sec:AltMin}, we present our alternating minimization method for solving \eqref{prob:srpcp}. The detailed steps for solving involved subproblems and the convergence analysis of the proposed algorithm are also presented in this section. Then, in Section~\ref{sec:acc}, we introduce our accelerated algorithm and discuss its detailed implementations.
	In Section~\ref{sec:num}, we conduct
	numerical experiments to evaluate the performance of our algorithms against the ADMM implemented by \cite{zhang2021square}. We conclude our paper in the last section.
	
	\medskip
	\noindent {\bf Notation}: Let $[n]:=\{1,2,\dots,n\}.$ Given any matrix $X$, we use $\|X\|$,  $\|X\|_F$, and $\|X\|_*$ to denote its spectral, Frobenius, and nuclear norm, respectively, and denote its vectorize $\ell_1$ norm by $\|X\|_1: = \sum_{i,j} |X_{ij}| $. The inner product of two matrices $X$ and $Y$ is given by $\langle X,Y \rangle = {\rm Tr}(X^TY)$.
	We use $\odot$ to denote the elementwise (Hadamard) product.
	Denote the Dirac distribution at a point $X$ by ${\bf 1}_{\{X\}}$.
	For any given vector $a \in \mathbb{R}^n$, we use $\|a\|_0$, $\|a\|_1$, $\|a\|_2$, and $\|a\|_{\infty}$ to denote its number of nonzero elements, $\ell_1$, $\ell_2$, and $\ell_{\infty}$ norms, respectively. Denote ${\rm sign}(a) := \{ v \in \mathbb{R}^n \,|\, v_i = 1 \mbox{ if } a_i>0;  v_i = -1 \mbox{ if } a_i<0;  v_i = 0 \mbox{ if } a_i=0\} $, and $|a|:=(|a_1|,\dots,|a_n|)$.
	${\rm Diag}(v)$ returns a square diagonal matrix with the elements of vector $v$ on the main diagonal.
	Let $\mathcal{B}_2:=\{v\,|\,\|v\|_2\leq 1\}$. Then, the subgradient of $f(x)=\|x\|_2$ is given by  $\partial f(x) = \{\frac{x}{\|x\|_2}\} \mbox{ if } x\neq 0; \partial f(x) = \mathcal{B}_2 \mbox{ if } x= 0$.
	The normal cone of $\mathbb{R}^n_+$ at $s\in\mathbb{R}^n_+$ is given by
	\begin{equation}\label{normal-cone}
		\mathcal{N}_{\mathbb{R}^n_+}(s) = \{v \in \mathbb{R}^n \,|\, \langle y-s,v \rangle \leq 0,\,\,\forall\, y\in\mathbb{R}^n_+\} = \{v \in \mathbb{R}^n \,|\, v_i = 0 \mbox{ if } s_i>0; v_i \leq 0 \mbox{ if } s_i=0\}.
	\end{equation}
	
	\section{Robust squared loss regression problems}
	As mentioned in the introduction, in the previous literature, e.g., \cite{candes2011robust,zhang2021square,zhou2010stable}, the regularized square root loss models are analyzed under the identifiability assumptions such as the incoherence assumption on the low-rank part and the uniformly sampling assumption of the support of the sparse part \cite{candes2011robust}.
	However, in many real applications, these assumptions are not satisfied or difficult to verify computationally \cite{chandrasekaran2011rank}. In this section, we take a different approach and interpret the regularized square root loss model as a distributionally robust optimization (DRO) problem.
	Specifically, we derive an equivalent DRO reformulation for the following optimization problem
	\begin{equation}\label{prob:srpcp-obs}
		\min_{L,\,S}\,\,\|L\|_* + \lambda\|S\|_1 + \mu \sqrt{\sum_{(i,j)\in \Omega_{\rm obs}} (L_{ij} + S_{ij} - D_{ij} )^2},
	\end{equation}
	which corresponds to a more general observation model compared to \eqref{obser-model}
	\begin{equation}\label{obser-model-2}
		D_{ij} = (L_0)_{ij} + (S_0)_{ij} + (Z_0)_{ij}, \quad (i,j) \in \Omega_{\rm obs}.	
	\end{equation}
	Here, the index set $\Omega_{\rm obs}$ is sampled uniformly at random from $
	\left\{ (j,k)\,|\, j\in[n_1],k\in[n_2] \right\}
	$.
	This model generalizes \eqref{obser-model} by allowing partially observed data, making it more applicable to scenarios where observations are incomplete.
	
	In fact, when samplings is performed without replacement and $\Omega_{\rm obs} = \left\{ (j,k)\,|\, j\in[n_1],k\in[n_2] \right\}$, \eqref{prob:srpcp-obs} simplifies to  \eqref{prob:srpcp}.
	The DRO interpretation  provides a new perspective on %
	the robustness properties of the regularized square root loss model \eqref{prob:srpcp-obs}.
	Moreover, without relying on the identifiability assumptions, it has the potential for us to tune the parameter $\mu$ based on the techniques of parameter selection derived from the perspective of DRO.
	Our new interpretation indicates that the regularized square root loss model \eqref{prob:srpcp-obs} is rather flexible and  can be applied to a broader range of applications.

	For the simplicity of notation, we denote $B_0 := [L_0;S_0]\in\mathbb{R}^{2n_1\times n_2}$. For any $B = [L;S]\in\mathbb{R}^{2n_1\times n_2}$ and $\lambda > 0$, we define a norm function $\|B\|_{\lozenge,\lambda} := \|L\|_* + \lambda\|S\|_1$. Following \cite{chu2021regularized}, the corresponding dual norm is given by $\|B\|_{\blacklozenge,\lambda} = \max\{\|L\|,\|S\|_{\infty}/\lambda\},\,\forall \, B = [L;S]\in\mathbb{R}^{2n_1\times n_2},\lambda>0$.
	To better illustrate the connection between the square root loss models and the distributionally robust optimization (DRO) problem, we adopt an observation model, as proposed by %
	\cite{bertsimas2018characterization}. Specifically, we consider a dataset of $n$ samples $(X_i,Y_i)$, which are assumed to be independent and identically distributed (i.i.d.), where each pair is drawn according to the following distribution:
	\begin{equation}\label{obser1}
		Y_i = \langle X_i,B_0 \rangle + \sigma \epsilon_i,\quad i \in [n].
	\end{equation}
	Here $\epsilon_i$'s are i.i.d. noises such that ${\rm E}(\epsilon_i) = 0$ and $ {\rm E}(\epsilon_i^2) = 1$,  $\sigma > 0$ is the unknown noise level, and $X_i$'s are i.i.d. random matrices sampled from
	\begin{equation}\label{set1}
		\Omega_0:=\left\{ [e_j(n_1)e_k^T(n_2);e_j(n_1)e_k^T(n_2)] \in \mathbb{R}^{2n_1\times n_2}\,|\, j\in[n_1],k\in[n_2] \right\},
	\end{equation}
	where $e_j(n_1)$ and $e_k(n_2)$ are the $j$ and $k$-th canonical basis vectors in $\mathbb{R}^{n_1}$ and $\mathbb{R}^{n_2}$, respectively.
	A specific instance of this setting is the observation model \eqref{obser-model-2}, where $\Omega_{\rm obs} \subseteq \Omega_0$ with $|\Omega_{\rm obs}|=n$.
	
	The next proposition follows directly from  Theorem~4 of  \cite{chu2021regularized}, %
	which establishes the equivalence between a DRO formulation of a squared loss linear regression problem  and a square root loss penalized model. In particular, the penalty level $\sqrt{\delta}$ is fully quantified by the ``radius'' $\delta$ of the uncertainty set. Before stating the proposition, we give some basic definitions in optimal transport; see Chapter~6 of \cite{villani2008optimal} for details. For any two probability measures $\mathbb{P}$ and $\mathbb{Q}$ in $\mathbb{R}^{2n_1\times n_2 + 1}$, $\Pi(\mathbb{P},\mathbb{Q})$ denotes the set of all joint probability measures on $\mathbb{R}^{2n_1\times n_2 + 1} \times \mathbb{R}^{2n_1\times n_2 + 1}$ whose marginals are $\mathbb{P}$ and $\mathbb{Q}$.
	For a given cost function $c:\,\mathbb{R}^{2n_1\times n_2 + 1} \times \mathbb{R}^{2n_1\times n_2 + 1} \to [0,\infty]$, where $c(u,v)$ is the cost for transporting one unit of mass from $u$ to $v$ and $c(u,u)=0$,
	the optimal transport cost between $\mathbb{P}$ and $\mathbb{Q}$ is defined as
	\begin{equation}\label{eq:transport-cost}
		\mathcal{D}_{c}(\mathbb{P},\mathbb{Q}) := \inf_{\pi \in \Pi(\mathbb{P},\mathbb{Q})}  \int c(u,v) d\pi(u,v).
	\end{equation}

	\begin{proposition} \label{prop:equiv1}
		Let the squared loss function be $ \ell (X,Y;B) = (Y - \langle X,B \rangle)^2,\,Y\in\mathbb{R},X,B\in\mathbb{R}^{2n_1\times n_2}$ and the cost function $c: \, \mathbb{R}^{2n_1\times n_2+1} \times \mathbb{R}^{2n_1\times n_2+1} \to [0, \infty] $ be defined by
		\begin{equation}\label{eq:def_c_dist}
			c ((X,Y),(\widetilde{X},\widetilde{Y})):= \left\{
			\begin{array}{ll}
				\|X - \widetilde{X}\|_{\blacklozenge,\lambda}^2, & \mbox{if } Y = \widetilde{Y}, \\[6pt]
				+\infty, & \mbox{otherwise}.
			\end{array} \right.
		\end{equation}
		Let $\mathbb{P}_n:= \frac{1}{n} \sum_{i=1}^{n} {\bf 1}_{\{(X_i,Y_i)\}}$ denote the empirical distribution.	Then it holds that
		\begin{equation}\label{eq:DRO}
			\inf_{B\in\mathbb{R}^{2n_1\times n_2}} \,\,\sup_{\mathbb{P}:\, \mathcal{D}_c( \mathbb{P},\mathbb{P}_n) \leq \delta }  \,{\rm E}_{\mathbb{P}} [\ell (X,Y;B) ]
			= \inf_{B\in\mathbb{R}^{2n_1\times n_2}} \, \left\{ \sqrt{{\rm E}_{\mathbb{P}_n} [\ell (X,Y;B) ]} + \sqrt{\delta}  \|B\|_{\lozenge,\lambda} \right\}^2.
		\end{equation}
	\end{proposition}
	
	We make an important remark on Proposition~\ref{prop:equiv1} about the uncertainty set $ \mathcal{U}:=\{\mathbb{P}\,|\, \mathcal{D}_c(\mathbb{P},\mathbb{P}_n) \leq \delta \}$ and the worst case loss. The uncertainty set $\mathcal{U}$ is centered at the empirical distribution $\mathbb{P}_n$, and the worst case loss accounts for all the probability measures that are plausible variations of $\mathbb{P}_n$.
	We observe that the marginal distribution of $\mathbb{P}\in\mathcal{U}$ on $Y$ should be the same as that of $\mathbb{P}_n$ since the cost function $c$ assigns infinite cost when $Y \neq \widetilde{Y}$.
	Denote the random variable $\xi := (X,Y)$, the samples $\hat{\xi}_i := (X_i,Y_i)$,  and the loss function $\ell(\xi) := \ell (X,Y;B)$. The worst case loss can be further written as:
	\begin{align}
		&\sup_{\mathbb{P}:\, \mathcal{D}_c(\mathbb{P},\mathbb{P}_n) \leq \delta }  \,{\rm E}_{\mathbb{P}} [\ell (\xi) ]\label{eq:loss}\\
		=& \sup_{\Pi,\mathbb{P}} \, \int \ell(\xi) \,\mathbb{P}(d\xi) \nonumber\\
		&{\rm s.t.} \int c(\xi,\xi') \Pi(d\xi,d\xi') \leq \delta,\Pi \mbox{ is a joint distribution of $\xi$ and $\xi'$ with marginals $\mathbb{P}$ and $\mathbb{P}_n$}\nonumber\\
		=& \sup_{\mathbb{P}^i}\, \frac{1}{n}\sum_{i=1}^n \int\ell(\xi) \,\mathbb{P}^i(d\xi)\quad{\rm s.t.}\quad \frac{1}{n}\sum_{i=1}^n \int c(\xi,\hat{\xi}_i)\,\mathbb{P}^i(d\xi) \leq\delta\nonumber\\
		=& \sup_{\mathbb{Q}^i}\, \frac{1}{n}\sum_{i=1}^n \int(Y_i - \langle\zeta,B\rangle )^2 \,\mathbb{Q}^i(d\zeta)\quad{\rm s.t.}\quad \frac{1}{n}\sum_{i=1}^n \int \|\zeta - X_i\|^2_{\blacklozenge,\lambda}\,\mathbb{Q}^i(d\zeta) \leq\delta\nonumber\\
		=&\sup_{\zeta_1,\dots,\zeta_n} \frac{1}{n}\sum_{i=1}^n{\rm E}\Big[\big(Y_i-\langle X_i,B\rangle + \langle \zeta_i,B\rangle \big)^2 \Big]  {\rm s.t.}
		\frac{1}{n} \sum_{i=1}^{n}{\rm E} \Big[\| \zeta_i\|_{\blacklozenge,\lambda}^2\Big] \leq \delta, \mbox{$\{\zeta_i\}$ are random variables.}\nonumber
	\end{align}
	
	We can see that the worst case loss is a perturbation of the squared loss with  small error terms $\zeta_i$.
	The next theorem, proved in Appendix \ref{app:dro_formulation}, illustrates the equivalence in Proposition~\ref{prop:equiv1} for the particular case \eqref{obser-model-2}.
	\begin{theorem}\label{thm:dro_formulation}
		Problem \eqref{prob:srpcp-obs} is equivalent to a DRO formulation of a squared loss problem
		\begin{equation}\label{op:robust}
			\begin{array}{cl}
				\displaystyle\inf_{L,S}\,\sup_{\xi_{jk},\zeta_{jk},(j,k)\in\Omega_{\rm obs}}\quad & \displaystyle \frac{1}{|\Omega_{\rm obs}|} \sum_{(j,k)\in\Omega_{\rm obs}}{\rm E}\Big[ \big((D- L- S)_{jk} + \langle \xi_{jk},L\rangle + \langle \zeta_{jk},S\rangle\big)^2\Big]
				\\[10pt]
				{\rm s.t.} & \displaystyle
				\sum_{(j,k)\in\Omega_{\rm obs}}{\rm E} \Big[\max\big\{\| \xi_{jk}\|^2,\| \zeta_{jk}\|_{\infty}^2/\lambda^2\big\}\Big] \leq \frac{1}{\mu^2},
			\end{array}
		\end{equation}
		where  $\xi_{jk},\zeta_{jk},(j,k)\in\Omega_{\rm obs}$ are random matrices in $\mathbb{R}^{n_1\times n_2}$.
	\end{theorem}
	
	Many studies, e.g., \cite{blanchet2019robust,shafieezadeh2015distributionally,gao2024wasserstein}, not only establish the fundamental connections between regularization problems and DRO problems, but also provide theoretical guidance on the selection of the ambiguity parameter $\delta$ and the regularization parameter $\mu$.	
	Although these results can not be directly applied to our models \eqref{prob:srpcp} and \eqref{op:robust} due to the absence of studies in the DRO literature on the specific loss function $c$ in \eqref{eq:def_c_dist}, they consistently demonstrate that the optimal choice of $\mu$ exhibits independence from the noise level, supporting the findings presented by %
	\cite{zhang2021square}.
	In our future studies, the appropriate selection of $\mu$ using the DRO framework will be further investigated.

	\section{An alternating minimization method for \eqref{prob:srpcp}}
	\label{sec:AltMin}
	As mentioned in the introduction, model \eqref{prob:srpcp} was solved via the two-block ADMM by %
	\cite{zhang2021square}.
	Specifically, by introducing slack variables $L_1$, $S_1$ and $Z$, %
	\cite{zhang2021square} rewrite model \eqref{prob:srpcp} into a linearly constrained optimization problem with a block separable objective \eqref{prob:srpcp_ADMM}.
	An important motivation for the introduction of these slack variables is to address the possible divergence of the multi-block ADMM \cite{chen2016direct}.
	However, as widely observed by \cite{boyd2011distributed}, this straightforward variable splitting technique, though it allows the direct application of the classic two-block ADMM algorithm, typically sacrifices computational efficiency. Meanwhile, it is widely acknowledged in the literature \cite{ghadimi2014optimal,xu2017accelerated,sabach2022faster}  that the performance of the ADMM heavily relies on the tuning of the involved penalty parameter.
	Although some simple general tuning rules appear in recent studies, e.g., \cite{wohlberg2017admm,tang2024self}, one generally cannot expect sound performance of using the ADMM for solving model \eqref{prob:srpcp_ADMM} by just mimicking these rules. Undoubtedly, heavy tuning is required, and most likely the tuning task is on a per-dataset basis.
	
	Here, we propose to solve model \eqref{prob:srpcp} directly via the alternating minimization (AltMin) method \cite{beck2015convergence,tseng2001convergence}. The detailed steps are given in Algorithm  \hyperref[algo:AltMin]{AltMin}. As one can observe, Algorithm  \hyperref[algo:AltMin]{AltMin} takes extremely simple updating rules, and more importantly, it is tuning-free.
	The main steps of Algorithm  \hyperref[algo:AltMin]{AltMin} are solving the nuclear norm and $\ell_1$ norm regularized subproblems \eqref{update-L} and \eqref{update-S}. We show in the next two subsections that these subproblems can be efficiently handled. Specifically,  these subproblems are shown to enjoy closed-form and easy-to-compute optimal solutions.
	We also conduct a comprehensive convergence analysis in subsection~\ref{subsec:convergence}.

	\begin{algorithm}[H] \label{algo:AltMin}
		\caption{{\bf AltMin}: An alternating minimization method for solving \eqref{prob:srpcp}}

		\vskip6pt
		\begin{algorithmic}
			\State \textbf{Initialization } Given the data matrix $D\in\mathbb{R}^{n_1\times n_2}$, parameters $\lambda>0$, $\mu>0$, choose any $L^0 \in \mathbb{R}^{n_1\times n_2}$.
			
			\For{ $i=0,1,\dots,$}
			
			\State \textbf{Step 1. } Compute
			\begin{equation}\label{update-S}
				S^{i+1} \in\arg\min_S \,\, \left\{  \lambda\|S\|_1 + \mu\|L^{i}+S-D\|_F \right\}.
			\end{equation}

			\State\textbf{Step 2. } Compute
			\begin{equation}\label{update-L}
				L^{i+1} \in\arg\min_L \,\, \left\{ \|L\|_* +  \mu\|L+S^{i+1}-D\|_F \right\}.
			\end{equation}
			\EndFor
		\end{algorithmic}
	\end{algorithm}

	\subsection{Explicit formula for updating $S$ in Algorithm  \hyperref[algo:AltMin]{AltMin}}\label{subsec:updateS}
	
	We first investigate the subproblem corresponding to $S$ in \eqref{update-S}. Though written in the matrix form, subproblem \eqref{update-S} can equivalently be rewritten in the following more abstract vector form:
	\begin{equation}\label{op:1}
		\min_{s\in\mathbb{R}^n} \,\, h(s):=\|s-a\|_2 + \tau \|s\|_1,
	\end{equation}
	where $a\in\mathbb{R}^n$ is the given data, and $\tau > 0$ is a given parameter. Note that $h$ is coercive, hence the optimal solution set to \eqref{op:1} is nonempty. Let $\bar{s}$ be an optimal solution to \eqref{op:1}. We first give some basic properties of $\bar{s}$ in Proposition~\ref{prop:1}. Roughly speaking, it states that the sign and order of $\bar{s}$ are ``consistent'' with those of $a$. See its proof in Appendix \ref{app:prop_1}.

	\begin{proposition}\label{prop:1}
		Let $\bar{s}$ be an optimal solution to \eqref{op:1}. For any $i\in[n]$, the following holds:
		\begin{itemize}
			\item[{\rm (i)}] If $a_i > 0$, then $0 \leq \bar{s}_i \leq a_i$;
			
			\item[{\rm (ii)}] If $a_i < 0$, then $a_i \leq \bar{s}_i \leq 0$;
			
			\item[{\rm (iii)}] If $a_i = 0$, then $\bar{s}_i = 0$.
		\end{itemize}
		
		Thus, it holds that
		\begin{itemize}
			\item[{\rm (iv)}]  $\tilde{s} \in \arg\min_{s\in\mathbb{R}^n} \,\, \big\{ \|s-|a|\|_2 + \tau \|s\|_1\big\}$ if and only if ${\rm sign}(a)\odot \tilde{s} \in \arg\min_{s\in\mathbb{R}^n} \,\, \big\{ \|s-a\|_2 + \tau \|s\|_1\big\}$;
			
			\item[{\rm (v)}]  $\tilde{s} \in \arg\min_{s\in\mathbb{R}^n} \,\, \big\{ \|s-Pa\|_2 + \tau \|s\|_1\big\}$ if and only if $P^T\tilde{s} \in \arg\min_{s\in\mathbb{R}^n} \,\, \big\{ \|s-a\|_2 + \tau \|s\|_1\big\}$, for any permutation matrix $P$.
			
		\end{itemize}
	\end{proposition}

	Proposition~\ref{prop:1} implies that we only need to focus on the case where the vector $a$ is nonnegative and sorted in descending order.
	Hence, without loss of generality, we make the following assumption on $a$:
	\begin{equation}\label{ass:1}
		a_1\geq a_2 \geq\dots\geq a_n\geq 0,\quad a\neq 0.
	\end{equation}
	Then, problem \eqref{op:1} is equivalent to
	\begin{equation}\label{op:2}
		\min_{s\in\mathbb{R}^n} \,\, \|s-a\|_2 + \tau \langle 1_n,s\rangle + \delta_{\mathbb{R}^n_+}(s),
	\end{equation}
	where $1_n\in\mathbb{R}^n$ is the vector of all ones.
	Next, by noting that  $\frac{1}{\sqrt{\|a\|_0}} \leq \frac{\|a\|_{\infty}}{\|a\|_2}$, where the equality holds if and only if all of the nonzero entries of $|a|$ are the same and equal to $\|a\|_{\infty}$, we derive in the following theorem the optimal solution to \eqref{op:1} under three different ranges of $\tau$, i.e.,
	\[ \tau \geq \frac{\|a\|_{\infty}}{\|a\|_2}, \quad 0 < \tau \leq \frac{1}{\sqrt{\|a\|_0}}, \, \mbox{ and } \, \frac{1}{\sqrt{\|a\|_0}} < \tau < \frac{\|a\|_{\infty}}{\|a\|_2}.\]

	\begin{theorem}\label{thm:1}
		Assume that the vector $a$ satisfies \eqref{ass:1}. It holds that:
		\begin{itemize}
			\item[{\rm (i)}]
			$\bar{s} = 0$ is optimal to \eqref{op:1} if and only if $\tau \geq \frac{\|a\|_{\infty}}{\|a\|_2}$;
			
			\item[{\rm (ii)}]
			$\bar{s} = a$ is optimal to \eqref{op:1} if and only if $0 < \tau \leq \frac{1}{\sqrt{\|a\|_0}}$;
			
			\item[{\rm (iii)}]
			suppose that $\tau \in (\frac{1}{\sqrt{\|a\|_0}}, \frac{\|a\|_{\infty}}{\|a\|_2})$.  Let $\bar{k}$ be the largest integer strictly less than $1/\tau^2$.
			Then $1\leq \bar{k}\leq \|a\|_0 -1 \leq n-1$, and there exists a unique $k\in\{1,2,\dots,\bar{k}\}$ such that
			\begin{equation}\label{search:k2}
				a_{k+1} \leq t_k:= \sqrt{\frac{a_{k+1}^2 + \dots + a_n^2}{\frac{1}{\tau^2}-k}} < a_k.
			\end{equation}
			Now,
			$
			\bar{s} := \max(a- t_k1_n,0) = [a_1- t_k;\dots;a_k- t_k;0;\dots;0]
			$
			is optimal to \eqref{op:1}.
		\end{itemize}
	\end{theorem}
	
	\begin{proof}
	Under the assumption on $a$, we know that  \eqref{op:1} is equivalent to \eqref{op:2}. It is sufficient to focus on \eqref{op:2}.
	We prove (i) by checking the optimality condition of \eqref{op:2} at $\bar{s}=0$:
	\[
	0\in -\frac{a}{\|a\|_2} + \tau 1_n + \mathcal{N}_{\mathbb{R}^n_+}(0)
	\overset{\eqref{normal-cone}}{\iff}
	0\in -\frac{a}{\|a\|_2} + \tau 1_n + \mathbb{R}^n_-
	\iff
	\tau - \frac{a_i}{\|a\|_2} \geq 0,\quad\forall\,i\in[n],
	\]
	which is further equivalent to $\tau \geq \frac{\|a\|_{\infty}}{\|a\|_2}.$
	For (ii), we can prove it similarly. Indeed, the optimality condition of \eqref{op:2} at $\bar{s}=a$
	is given by
	$
	0\in \mathcal{B}_2 + \tau 1_n + \mathcal{N}_{\mathbb{R}^n_+}(a),
	$
	which, by \eqref{normal-cone}, is equivalent to
	\[
	0\in \mathcal{B}_2 + \tau 1_n +  \left[0_r; \mathbb{R}^{n-r}_- \right]
	\iff
	\mathcal{B}_2 \cap \left\{v\in\mathbb{R}^n\mid v_i=\tau,i=1,\dots,r;v_i\leq \tau,i=r+1,\dots,n\right\} \neq \emptyset,
	\]
	with $r:=\|a\|_0$. Now, it is not difficult to observe that the last relation holds if and only if
	$
	[\tau 1_r; 0_{n-r}]\in\mathcal{B}_2
	$, i.e., $\sqrt{r} \tau \leq 1.$
	Now we turn to prove (iii), and start by showing the existence of $k$ in \eqref{search:k2}. By the definition of $\bar{k}$ and the assumption on $\tau$, it holds that
	\[0 < \frac{1}{\tau^2} - 1 \leq \bar{k} < \frac{1}{\tau^2}.\] Since $\tau > \frac{1}{\sqrt{r}}$, we know that $\bar{k} < r$, and thus $1 \le \bar{k} \leq r - 1 \leq n-1$.  Then, the range of $\tau$ implies that
	\begin{equation*}
		\begin{array}{rl}
			\displaystyle \Big[a_{i+1} - t_i \Big]\Big|_{i=0}
			& \displaystyle = a_1 - \sqrt{\frac{a_1^2+\cdots+a_n^2}{\frac{1}{\tau^2}}}
			=  \|a\|_{\infty} - \tau \|a\|_2 > 0,\\[0.5cm]
			\displaystyle \Big[a_{i+1} - t_i \Big]\Big|_{i=\bar{k}}
			& \displaystyle =a_{\bar{k}+1} - \sqrt{\frac{a_{\bar{k}+1}^2+\cdots+a_n^2}{\frac{1}{\tau^2}-\bar{k}}}
			\leq a_{\bar{k}+1} - \sqrt{a_{\bar{k}+1}^2+\cdots+a_n^2} \leq {0.}
		\end{array}
	\end{equation*}
	Therefore, there {exists}
	$k \in \{1,\dots,\bar{k}\}$ such that
	\begin{equation}\label{eq:pf_unique}
		a_k - t_{k-1}  > 0 \mbox{ and } a_{k+1} - t_k  \leq 0.
	\end{equation}
	It remains to show $t_k < a_k$ in \eqref{search:k2}.
	The inequality $ a_k - t_{k-1}  > 0$ and the definition of $t_{k-1}$ imply that
	\begin{equation}\label{eq:akgtk}
		a_k^2 > \frac{a_k^2}{\frac{1}{\tau^2}-k+1} + \frac{a_{k+1}^2+\cdots+a_n^2}{\frac{1}{\tau^2}-k+1}, \quad \mbox{or equivalently}, \quad a_k^2 > \frac{a_{k+1}^2+\cdots+a_n^2}{\frac{1}{\tau^2} - k} = t_k^2.
	\end{equation}
	Therefore, $a_k > t_k$ and the existence of $k$ in \eqref{search:k2} is proved.

	Now, we argue that such a $k$ is unique. In fact, we can show that $k$ satisfying \eqref{eq:pf_unique} is unique.
	To prove this by contradiction, suppose that there exist two distinct integers $k_1$ and $k_2$ such that $1\leq k_1 <  k_2\leq \bar{k}$, $a_{k_1} - t_{k_1-1}  > 0$, $ a_{k_1+1} - t_{k_1}  \leq 0$, $a_{k_2} - t_{k_2-1}  > 0$, and $ a_{k_2+1} - t_{k_2}  \leq 0$. Namely, we have that
	\begin{equation*}
		\Big[a_{i+1} - t_i \Big]\Big|_{i=k_1} \leq 0, \quad
		\Big[a_{i+1} - t_i \Big]\Big|_{i=k_2-1}  > 0.
	\end{equation*}
	Therefore, there exists $p\in\{ k_1+1,\dots,k_2-1\}$ such that $a_p - t_{p-1}  \leq 0$  and $ a_{p+1} - t_p  > 0$. Similar to \eqref{eq:akgtk}, we can deduce from $a_p  \leq t_{p-1}$ that $a_p  \leq t_p$. This further implies $a_p  \leq t_p < a_{p+1}$, which contradicts the sorted assumption  of $a$ in \eqref{ass:1}. Thus, we have shown the uniqueness of $k$.
	
	Lastly, we prove the optimality of $\bar s$ by  examining whether the following optimality condition holds at $\bar{s}$:
	\begin{equation}\label{opt-cond}
		0\in \frac{\bar{s} - a}{\|\bar{s} - a\|_2} + \tau 1_n + \mathcal{N}_{\mathbb{R}^n_+}(\bar{s})
		\overset{\eqref{normal-cone}}{=}\frac{\bar{s} - a}{\|\bar{s} - a\|_2} + \tau 1_n + [0_k; \mathbb{R}^{n-k}_-].
	\end{equation}
	By simple computations, we have that
	\[
	\|\bar{s} - a\|_2^2
	= k t_k^2 + a_{k+1}^2 + \dots + a_n^2
	= \frac{a_{k+1}^2 + \dots + a_n^2}{1-k\tau^2} = \frac{t_k^2}{\tau^2},
	\]
	and thus
	\[
	\frac{\bar{s} - a}{\|\bar{s} - a\|_2} + \tau 1_n
	= \left[0_k; \tau-\tau\frac{a_{k+1}}{t_k};\dots;  \tau-\tau\frac{a_{n}}{t_k}\right].
	\]
	Therefore, \eqref{opt-cond} holds due to the fact that
	$
	1-{a_{n}}/{t_k} \geq \dots \geq  1-{a_{k+1}}/{t_k}  \overset{\eqref{search:k2}}{\geq} 0.
	$
	\end{proof}

	\begin{remark}
		Following the above theorem, we have two trivial observations: 1) If $\tau \geq 1$, then $0$ is optimal to \eqref{op:1} (since $\|a\|_{\infty} \leq \|a\|_2$); 2) If $\tau \leq 1/\sqrt{n}$, then $a$ is optimal to \eqref{op:1} (since $\|a\|_0 \leq n$).
	\end{remark}
	
	In the following proposition, we show problem \eqref{op:1} admits a unique solution
	unless $\tau = 1/\sqrt{\|a\|_0}$. Its proof can be found in Appendix \ref{app:unique}.
	\begin{proposition}\label{pro:unique}
		Assume that the vector $a$ satisfies \eqref{ass:1}. Then,
		\begin{itemize}
			\item [{\rm (i)}] if $\tau \neq \frac{1}{\sqrt{\|a\|_0}}$, then problem \eqref{op:1}  has a unique optimal solution;
			
			\item [{\rm (ii)}] if  $\tau = \frac{1}{\sqrt{\|a\|_0}}$, then for any $\varepsilon$ satisfying $0\leq \varepsilon \leq \min\{a_i\,|\,a_i > 0\}$, $a^\varepsilon:= a - \varepsilon\,{\rm sign}(a) \odot 1_n$
			solves \eqref{op:1}.
		\end{itemize}
	\end{proposition}

	\subsection{Explicit formula for updating $L$ in Algorithm  \hyperref[algo:AltMin]{AltMin}}\label{subsec:updateL}
	
	Now we focus on the subproblem for computing $L$ in \eqref{update-L}.
	For the later developments, \eqref{update-L} is written in the following abstract form:
	\begin{equation}\label{op:L}
		\min_L\,\, \|L - A\|_F + \rho\|L\|_*,
	\end{equation}
	where the matrix $A\in\mathbb{R}^{n_1\times n_2}$  and the parameter $\rho > 0$ are given. Based on the singular value decomposition, we derive in this subsection the optimal solution to \eqref{op:L}
	based on the results obtained in the previous subsection.
	
	For any nonnegative and nonzero vector $a\in\mathbb{R}^{n_2}$, the operator $d_{\rho}:\,\mathbb{R}^{n_2}\to\mathbb{R}^{n_2}$ is defined as:
	\begin{equation}\label{def-d}
		d_{\rho}(a) =
		\begin{cases}
			0, & \mbox{ if } \rho\geq \frac{\|a\|_{\infty}}{\|a\|_2},\\
			\arg\min_s \,\, \{\|s-a\|_2 + \rho \|s\|_1 \}, & \mbox{ if }  \frac{1}{\sqrt{\|a\|_0}} < \rho < \frac{\|a\|_{\infty}}{\|a\|_2},\\
			a, & \mbox{ if } \rho \leq \frac{1}{\sqrt{\|a\|_0}}.
		\end{cases}
	\end{equation}
	
	\begin{proposition}\label{pro:diagA}
		Assume that $A =
		\begin{pmatrix}
			{\rm Diag}(\sigma)  \\
			0
		\end{pmatrix}
		$
		with $\sigma:=(\sigma_1,\dots,\sigma_{n_2})^T$ and   $\sigma_1 \geq \dots \geq \sigma_{n_2} \ge 0$.
		Then
		$\begin{pmatrix}
			{\rm Diag}(d_{\rho}(\sigma))  \\
			0
		\end{pmatrix}$
		is optimal to \eqref{op:L}.
	\end{proposition}
	Its proof is documented in Appendix \ref{app:diagA}.
	In the next theorem, we consider the general case.
	
	\begin{theorem}\label{thm:update_L}
		Let $A = U\Sigma V^T$ be a singular value decomposition, where $U \in \mathbb{R}^{n_1\times n_1}$ and $V\in\mathbb{R}^{n_2\times n_2}$ are orthogonal matrices, $\Sigma =
		\begin{pmatrix}
			{\rm Diag}(\sigma) \\
			0
		\end{pmatrix} \in \mathbb{R}^{n_1\times n_2}
		$
		with $\sigma:=(\sigma_1,\dots,\sigma_{n_2})^T$,  $\sigma_1 \geq \dots \geq \sigma_{n_2} \ge 0$. Then $U
		\begin{pmatrix}
			{\rm Diag}(d_{\rho}(\sigma)) \\
			0
		\end{pmatrix}
		V^T$   
		is optimal to \eqref{op:L}, where $d_{\rho}$ is defined by \eqref{def-d}.
	\end{theorem}
	\begin{proof}
	The desired result follows directly from the orthogonal invariance of  $\|\cdot\|_F$ and $\|\cdot\|_*$ and Proposition \ref{pro:diagA}.
	\end{proof}

	\subsection{Convergence analysis}\label{subsec:convergence}
	In this subsection, we study the convergence properties of the proposed method. There is a vast literature on the convergence analysis of the alternating minimization method. We refer the readers to the work of \cite{tseng2001convergence} for a comprehensive review. Here, we establish the convergence of Algorithm  \hyperref[algo:AltMin]{AltMin} for solving \eqref{prob:srpcp} based on results of \cite{tseng2001convergence}.
	
	We start by introducing relevant definitions, which are mainly taken from \cite{tseng2001convergence}.
	Let $\mathcal{X},\mathcal{X}_1,\mathcal{X}_2$ be finite dimensional Euclidean spaces.
	For any extended real valued function $h:\,\mathcal{X} \to (-\infty,+\infty]$, the effective domain of $h$ is denoted as ${\rm dom} \, h$.
	For any $x\in{\rm dom}\,h$ and any $d\in\mathcal{X}$, the (lower) directional derivative of $h$ at $x$ in the direction $d$ is defined as
	$ h'(x;d)=\lim\inf_{t \downarrow 0}\, [h(x+td) - h(d)]/t. $
	Point $z$ is said to be a stationary point of $h$ if $z\in{\rm dom}\,h$ and $h'(z;d)\geq 0$ for all $d$.
	For any $f:\,\mathcal{X}_1\times\mathcal{X}_2 \to (-\infty,+\infty]$, %
	$(x_1,x_2)$ is said to be a coordinatewise minimum point of $f$ if $(x_1,x_2)\in{\rm dom}\,f$, $f(x_1+d_1,x_2) \geq f(x_1,x_2)$, and $f(x_1,x_2+d_2) \geq f(x_1,x_2)$,  for all $d_1\in\mathcal{X}_1$ and $d_2\in \mathcal{X}_2$.
	The function $f$ is said to be regular at $(x_1,x_2)$ if $f'((x_1,x_2);(d_1,d_2)) \geq 0$ whenever $f'((x_1,x_2);(d_1,0)) \geq 0 $ and $ f'((x_1,x_2);(0,d_2)) \geq 0$, for all $d_1\in\mathcal{X}_1$ and $d_2\in\mathcal{X}_2$.
	We know from Section~3 of \cite{tseng2001convergence} that a coordinatewise minimum point $(x_1,x_2)$ of $f$ is a stationary point of $f$ whenever $f$ is regular at $(x_1,x_2)$.
	For \eqref{prob:srpcp}, we denote $f_0:\,\mathbb{R}^{n_1\times n_2}\times \mathbb{R}^{n_1\times n_2} \to \mathbb{R}$ by $$
	f_0(L,S):=\mu\|L+S-D\|_F,\,\,(L,S)\in\mathbb{R}^{n_1\times n_2}\times\mathbb{R}^{n_1\times n_2}.
	$$
	We further denote $f_1:\,\mathbb{R}^{n_1\times n_2} \to \mathbb{R}$ by $f_1(\cdot):=\|\cdot\|_*$ and $f_2:\,\mathbb{R}^{n_1\times n_2} \to \mathbb{R}$ by $f_2(\cdot):=\lambda\|\cdot\|_1$, and $f:\,\mathbb{R}^{n_1\times n_2}\times \mathbb{R}^{n_1\times n_2} \to \mathbb{R}$ by
	\begin{equation}\label{def-f}
		f(L,S):=f_0(L,S) + f_1(L) + f_2(S),\,\,(L,S)\in\mathbb{R}^{n_1\times n_2}\times\mathbb{R}^{n_1\times n_2}.
	\end{equation}

	\begin{proposition}\label{prop:9}
		The function $f$ defined by \eqref{def-f} is regular at $(L,S)$ whenever $L + S \neq D$.
	\end{proposition}
	\begin{proof}
	Take any $(L,S)$ satisfying $L + S \neq D$. It follows from simple calculations that
	\begin{equation}\label{eq:pf1}
		f_0'((L,S);(d_L,d_S)) = f_0'((L,S);(d_L,0)) + f_0'((L,S);(0,d_S)),\,\,\forall\,d_L,d_S,
	\end{equation}
	and thus $f'((L,S);(d_L,d_S)) = f'((L,S);(d_L,0)) + f'((L,S);(0,d_S)),\,\,\forall\,d_L,d_S$.
	\end{proof}
	
	The proof of the following convergence theorem closely follows the approach in  Theorem~4.1(c) of \cite{tseng2001convergence}. For the sake of completeness, we include it in Appendix \ref{app:convergence}.
	\begin{theorem}\label{thm:convergence}
		The sequence $\{(L^i,S^i)\}$ generated by Algorithm~ \hyperref[algo:AltMin]{AltMin} is bounded, and every cluster point $(\overline{L},\overline{S})$ of $\{(L^i,S^i)\}$ is a coordinatewise minimum point of $f$. In addition, if $f$ is non-overfitting at $(\overline{L},\overline{S})$, i.e., $\overline{L} + \overline{S} \neq D$, then $(\overline{L},\overline{S})$ is a minimum point of $f$, i.e., an optimal solution to \eqref{prob:srpcp}.
	\end{theorem}

	\begin{remark}
		We note that the non-overfitting assumption in Theorem \ref{thm:convergence} arises naturally from the inherent random noise \(Z_0\) in the observation model \eqref{obser-model}. Under the standard identifiability assumptions of \cite{candes2011robust} and \cite{zhou2010stable},
		\cite{zhang2021square} demonstrate both theoretically and numerically that the optimal solution \((\overline{L}, \overline{S})\) to \eqref{prob:uspcp}, with \(\lambda = 1/\sqrt{n_1}\) and \(\mu = \sqrt{n_2/2}\), provides a good approximation of the ground truth \((L_0, S_0)\) in model \eqref{obser-model}. Consequently, it is safe to assume that \(\|D - \overline{L} - \overline{S}\|_F\) is  not zero, given the presence of $Z_0$.		
	\end{remark}
	
	\section{Acceleration of the AltMin method via the BM-decomposition}
	\label{sec:acc}
	From Theorem \ref{thm:update_L} in Section~\ref{subsec:updateL}, we observe that computing $L^i$ at each iteration of Algorithm  \hyperref[algo:AltMin]{AltMin}, requires performing a full singular value decomposition (SVD). This decomposition can be computationally expensive, especially under the high-dimensional setting. In this section, we propose to alleviate this computational burden by leveraging the low-rankness of the iterates $\{L^i\}$ enforced by the nuclear norm regularizer. Particularly, we propose to employ the Burer-Monteiro (BM) decomposition \cite{Burer2003nonlinear,Burer2005local} to represent a low-rank matrix. By combining with the  variational formulation of the nuclear norm  \cite{rennie2005fast}, this results in a smooth nonconvex reformulation of problem \eqref{op:L}. This reformulation enables us to replace full SVD with partial SVD across the iterations, which significantly reduces the computational costs especially for large-scale problems.
	Recent advances in GPU technology have enabled GPU acceleration for partial SVD computation, see, e.g., \cite{struski2024efficient}, which represents a promising direction for our future research work.
	
	We first recall the variational formulation \cite{recht2010guaranteed,rennie2005fast} of the nuclear norm in the following lemma.
	\begin{lemma}%
		\label{lemma:variational nuclear norm}
		Given any $X\in\mathbb{R}^{n_1\times n_2}$, we have
		\begin{equation}
			\label{eq:variational nuclear norm}
			\|X\|_*=\min_{U,V\atop X=UV^T}\|U\|_F\|V\|_F=\min_{U,V\atop X=UV^T}\frac12\left(\|U\|_F^2+\|V\|_F^2\right).
		\end{equation}
		Let the SVD of $X$ be
		\begin{equation}
			X= H \Sigma W^T=\sum_{i=1}^r \sigma_ih_iw_i^T.
		\end{equation}
		Here $r={\rm rank}(X)$, $\Sigma=\operatorname{Diag}(\sigma_1,\dots,\sigma_r)$ and $\sigma_i>0$ are the singular values. The matrices $H=[h_1,\dots,h_r]\in\mathbb{R}^{n_1\times r}$ and $W=[w_1,\dots,w_r]\in\mathbb{R}^{n_2\times r}$ are orthonormal. Then, the minima of the right rand of \eqref{eq:variational nuclear norm} are exactly those
		\begin{equation}
			U^*:=\left[\sqrt{\sigma_{\tau(1)}}h_{\tau(1)},\dots,\sqrt{\sigma_{\tau(r)}}h_{\tau(r)}\right],\quad V^*:=\left[\sqrt{\sigma_{\tau(1)}}w_{\tau(1)},\dots,\sqrt{\sigma_{\tau(r)}}w_{\tau(r)}\right],
		\end{equation}
		where $\tau$ is an arbitrary permutation of $\{1,\dots,r\}$.
	\end{lemma}
	
	Based on Lemma \ref{lemma:variational nuclear norm}, with a prior guess $k\le n_2$ of the rank of an optimal solution to \eqref{op:L}, the problem \eqref{op:L} can be reformulated as follows:
	\begin{equation}
		\label{op:UV1}
		\min_{U\in\mathbb{R}^{n_1 \times k},\atop
			V\in\mathbb{R}^{n_2 \times k}} f_{\rm pen}(U,V):= \|UV^T-A\|_F+\frac{\rho}{2}\left(\|U\|_F^2+\|V\|_F^2\right).
	\end{equation}
	It is not difficult to notice that the optimal solution set to \eqref{op:UV1}, denoted by $\Omega_{\rm pen}$, is nonempty since the objective function is coercive. Proposition~\ref{proposition:solver_UV} summarizes the relations between  \eqref{op:UV1} and \eqref{op:L}.
	\begin{proposition}\label{proposition:solver_UV}
		Let $L^*$ be an optimal solution to \eqref{op:L}. Assume $k\ge {\rm rank}(L^*)$, and let the  SVD of $L^*$ be $L^*=H^*\Sigma^*(W^*)^T$, where $H^*\in \mathbb{R}^{n_1\times k}$, $\Sigma^*\in \mathbb{R}^{k\times k}$, and $W^*\in \mathbb{R}^{n_2 \times k}$.  Then, \begin{equation}
			U^*=H^*(\Sigma^*)^{\frac12},V^*=W^*(\Sigma^*)^{\frac12}
		\end{equation}
		is an optimal solution to \eqref{op:UV1}.
		Conversely, if $U^*\in\mathbb{R}^{n_1\times k},V^*\in\mathbb{R}^{n_2\times k}$ is an optimal solution to \eqref{op:UV1}, then $U^*(V^*)^T$ is an optimal solution to \eqref{op:L}.
	\end{proposition}

	Next, we show that for any given positive integer $k$, it is possible to derive a closed-form optimal solution for the possibly nonconvex optimization problem \eqref{op:UV1}. For this purpose, we take the same strategy as in subsection~\ref{subsec:updateL} by first exploring the special structure of the optimal solutions to \eqref{op:UV1} with a simple diagonal matrix $A$.

	\begin{proposition}\label{prop:optimal_UV}
		Assume that $A =
		\begin{pmatrix}
			{\rm Diag}(a)  \\
			0
		\end{pmatrix}\in\mathbb{R}^{n_1\times n_2}$ with $a \in \mathbb{R}^{n_2}$ and $a_1\ge a_2\ge \ldots \ge a_{n_2}\ge 0$. Then, there exists an optimal solution $(U^*, V^*) \in \Omega_{\rm pen}$ to \eqref{op:UV1} taking the form
		\[ U^* =
		\begin{pmatrix}
			{\rm Diag}(\xi^*)  \\
			0
		\end{pmatrix}\in\mathbb{R}^{n_1\times k} \mbox{ and } \, V^* =
		\begin{pmatrix}
			{\rm Diag}(\xi^*)  \\
			0
		\end{pmatrix} \in \mathbb{R}^{n_2\times k}
		\] with $\xi^*\in \mathbb{R}^{k}$ satisfying $\xi_1 \ge \ldots \ge \xi_k \ge 0$.
	\end{proposition}
	\begin{proof}
	We first note that $U_{n_2+1:n_1,:} = 0$ for all $(U,V) \in \Omega_{\rm pen}$ since $A_{n_2+1:n_1,:} = 0$. Hence, without loss of generality, we can assume that $n_1 = n_2$. Then, $A = {\rm Diag}(a)$.
	
	For any $U,V\in\mathbb{R}^{n_2\times k}$, it follows from \eqref{eq:variational nuclear norm} that
	$\frac{1}{2}\left(\|U\|_F^2 + \|V\|_F^2\right)\geq \|UV^T\|_*$.
	This, together with Corollary 7.3.5(b) of \cite{horn2012matrix}, implies that
	
	\begin{equation}
		\label{eq:fpenoverfaux}
		f_{\rm pen}(U,V) \ge \|\sigma(UV^T) - a\|_2 + \rho \|\sigma(UV^T)\|_1, \quad \forall\ U, V\in\mathbb{R}^{n_2\times k}.
	\end{equation}
	Here, the operator $\sigma: \mathbb{R}^{n_2\times n_2} \to \mathbb{R}^{n_2}$ extracts the singular values of a matrix and arranges them in descending order.
	Based on the above inequality, we consider the %
	auxiliary problem of \eqref{op:UV1}:
	\begin{equation}
		\label{prob:aux}
		\min_{U, V\in\mathbb{R}^{n_2\times k}} f_{\rm aux}(U,V) := \|\sigma(UV^T) - a\|_2 + \rho \|\sigma(UV^T)\|_1.
	\end{equation}
	It is not difficult to see that $f_{\rm aux}$ is coercive, and thus, the optimal solution set to \eqref{prob:aux}, %
	denoted by $\Omega_{\rm aux}$, is nonempty. Let $(\widehat U, \widehat V)\in \Omega_{\rm aux}$ be an arbitrary optimal solution to \eqref{prob:aux}.
	Consider the following SVD %
	of $\widehat U \widehat V^T$ as
	\[\widehat U \widehat V^T = \widehat H \widehat \Sigma \widehat W^T \mbox{ with } \widehat H, \widehat W^T \in \mathbb{R}^{n_2\times k}, \, \widehat\Sigma \in \mathbb{R}^{k\times k}.\]
	Now, let
	$
	\widetilde U = \begin{pmatrix}
		\widehat \Sigma^{\frac{1}{2}}  \\
		0
	\end{pmatrix}\in\mathbb{R}^{n_2 \times k}
	$ and $\widetilde V = \widetilde U$. It is obvious that
	$\sigma(\widehat U\widehat V^T) = \sigma(\widetilde U \widetilde V^T)$. Then, we see that $f_{\rm aux}(\widehat U, \widehat V) = f_{\rm aux}(\widetilde U, \widetilde V)$ and thus $(\widetilde U, \widetilde V)\in \Omega_{\rm aux}$. Moreover, by observing that
	\[
	\frac{1}{2}\left(\|\widetilde U\|_F^2 + \|\widetilde V\|_F^2 \right) = \|\sigma(\widetilde U\widetilde V)\|_1 \mbox{ and } \|\widetilde U \widetilde V^T - A\|_F = \|\sigma(\widetilde U\widetilde V^T) - a\|_2,
	\]
	we have that $f_{\rm pen}(\widetilde U, \widetilde V) = f_{\rm aux}(\widetilde U, \widetilde V)$. Then, it holds from
	\eqref{eq:fpenoverfaux} and $(\widetilde U, \widetilde V)\in \Omega_{\rm aux}$ that
	\[
	f_{\rm pen}(U,V) \ge f_{\rm aux}(U,V) \ge f_{\rm aux}(\widetilde U, \widetilde V) = f_{\rm pen}(\widetilde U, \widetilde V), \quad \forall\ U, V\in\mathbb{R}^{n_2\times k},
	\]
	which implies that $(\widetilde U, \widetilde V)\in \Omega_{\rm pen}$. Now, we can set $\xi^*$ as the vector of diagonal entries of $\widehat \Sigma^{1/2}$ %
	and complete the proof.
	\end{proof}
	
	Inspired by Proposition \ref{prop:optimal_UV}, in the subsequent analysis, we focus on the following %
	problem:
	\begin{equation}
		\label{op:uv}
		\min_{s\in\mathbb{R}^k,s\geq 0}  \sqrt{\|s- w \|_2^2+ c^2}+\rho \left\langle  s,1_k \right\rangle,
	\end{equation}
	where $c\ge 0$ and $w \in \mathbb{R}^k$ with $w_1 \ge \cdots \ge w_k \ge 0$ are given data. In fact, when $A$
	is diagonal, we can obtain a special optimal solution pair $(U^*,V^*)\in \Omega_{\rm pen}$ %
	by solving \eqref{op:uv} with $w = a_{1:k}$ and $c =\sqrt{\sum_{i=k+1}^{n_2} a_{i}^2}$. %
	With this observation, we derive %
	in the following theorem an optimal solution to \eqref{op:UV1} %
	for a general matrix $A$.
	
	\begin{theorem}\label{thm:opt_UV_general}
		Let $A = H \Sigma W^T$ be its SVD, where $H \in \mathbb{R}^{n_1\times n_1}$ and $W\in\mathbb{R}^{n_2\times n_2}$ are orthogonal matrices, $\Sigma =
		\begin{pmatrix}
			{\rm Diag}(a) \\
			0
		\end{pmatrix} \in \mathbb{R}^{n_1\times n_2}
		$
		with $a\in \mathbb{R}^{n_2}$ and  $a_1 \geq \dots \geq a_{n_2} \ge 0$.
		Let $s^* \in \mathbb{R}^k$ be an optimal solution to \eqref{op:uv} with $w = a_{1:k}$, and $c = \sqrt{\sum_{i=k+1}^{n_2} a_i^2}$. Then, an optimal solution to \eqref{op:UV1} is
		\begin{equation}\label{eq:UV}
			(U^*, V^*) = \left(
			H \begin{pmatrix}
				{\rm Diag}(\sqrt{s^*}) \\
				0_{(n_1-k)\times k}
			\end{pmatrix}, \, W \begin{pmatrix}
				{\rm Diag}(\sqrt{s^*}) \\
				0_{(n_2-k)\times k}
			\end{pmatrix}
			\right) \in \Omega_{\rm pen},
		\end{equation}
		where $\sqrt{s^*}:= (\sqrt{s_1},\ldots, \sqrt{s_k})^T$.
	\end{theorem}
	\begin{proof}
	From the orthogonal invariance of the Frobenius norm, we have that
	\[
	f_{\rm pen}(U,V) = \|H^T U V^T W - \Sigma\|_F + \frac{\rho}{2}(\|H^T U\|_F^2 + \|W^T V\|_F^2).
	\]
	The desired result then follows from Proposition \ref{prop:optimal_UV}.
	\end{proof}

	Theorem \ref{thm:opt_UV_general} indicates that an optimal solution to \eqref{op:UV1} can be obtained by solving a simplified problem \eqref{op:uv}.
	The following theorem gives an explicit formula of an optimal solution to \eqref{op:uv}, which can be regarded as a generalization of Theorem \ref{thm:1}. When $c=0$, Theorem \ref{thm:overparameter} reduces to Theorem \ref{thm:1}, as problem \eqref{op:uv} reduces to problem \eqref{op:2}.
	Its proof is similar to that of Theorem \ref{thm:1} and is thus documented in Appendix \ref{sec:proof_thm_7} for the clarity and coherence in presentation.
	
	\begin{theorem}
		\label{thm:overparameter}
		Assume that the vector $w$ and the constant $c$ in \eqref{op:uv} satisfy
		\begin{equation}
			\label{assum:uv}
			w_1\geq w_2\geq\dots\geq w_k\geq0, \quad w\neq0, \quad c\geq0.
		\end{equation}
		Denote $\ell=\|w\|_0$.
		Then, it holds that $d_\rho(w,c)$ defined by
		\begin{equation}\label{eq:lowrankd}
			d_{\rho}(w,c)= \begin{cases}
				0, &\text{if}\quad \frac{\|w\|_\infty}{\sqrt{\|w\|_2^2 +c^2}}\leq \rho,\\
				\max(w-t_i  1_k,0),&\text{if}\quad \frac{w_\ell}{\sqrt{\ell w_\ell^2+c^2}} <\rho
				<\frac{\|w\|_\infty}{\sqrt{\|w\|_2^2+c^2}},\\
				w-\rho\sqrt{\frac{ c^2}{1-k\rho^2}}1_k,&\text{if }\quad 0< \rho\leq \frac{w_\ell}{\sqrt{\ell w_\ell^2+c^2}},
			\end{cases}
		\end{equation}
		is  optimal to \eqref{op:uv}, where the index $i \in \{1,2,\ldots, \ell - 1\}$ is the unique integer satisfying
		\begin{equation}\label{search:rk2}
			w_{i+1} \leq t_i:= \sqrt{\frac{w_{i+1}^2 + \dots + w_k^2+c^2}{\frac{1}{\rho^2}-i}} < w_i.
		\end{equation}
	\end{theorem}
	
	\begin{remark}\label{remark:ineq}
		We shall mention that under the assumption \eqref{assum:uv}, the inequality
		${w_\ell}/{\sqrt{\ell w_\ell^2 + c^2}} \leq {\|w\|_\infty}/{\sqrt{\|w \|_2^2 + c^2}}$ holds naturally.
		Indeed, since $\|w\|_2^2 \leq \ell w_1^2$, we have \[\frac{\|w\|_\infty}{\sqrt{\|w\|_2^2 + c^2}} \geq \frac{\|w\|_\infty}{\sqrt{l w_1^2 + c^2}}
		= \frac{1}{\sqrt{\ell + \frac{c^2}{w_1^2}}} \geq \frac{1}{\sqrt{\ell + \frac{c^2}{w_\ell^2}}} = \frac{w_\ell}{\sqrt{\ell w_\ell^2 + c^2}}
		.\]
	\end{remark}

	One potential difficulty of our above BM-decomposition approach is to obtain a reasonable estimation of the parameter $k$ in \eqref{op:UV1}.
	Proposition \ref{proposition:solver_UV} establishes a sound relation between \eqref{op:UV1} and \eqref{op:L} only under the assumption that $k\ge {\rm rank}(L^*)$.
	However, the rank of $L^*$ is generally unavailable.
	In the following discussions, we show that it is possible to assess the correctness of $k$, or equivalently the legitimacy of problem \eqref{op:UV1}, without knowing the true rank of  $L^*$.
	We begin by observing that the correctness of $k$ can be verified  by checking whether the optimal solution to \eqref{op:uv} can be extended to yield an optimal solution to problem \eqref{op:2}.
	Indeed, Theorem \ref{thm:update_L} indicates that the rank of $L^*$ is determined by the optimal solution to the corresponding problem \eqref{op:2}, and Theorem \ref{thm:opt_UV_general} shows that the optimal solution to \eqref{op:UV1} can be obtained by using the matrix decomposition and solving the  optimization problem \eqref{op:uv}.
	In the following proposition, we provide a direct and readily assessable criterion for determining whether the optimal solution to \eqref{op:uv} can be extended to be an optimal solution to problem \eqref{op:2}.
	Its proof is straightforward yet tedious, and for clarity, the details are deferred to  Appendix \ref{sec:proof_prop_8}.
	
	\begin{proposition}\label{prop:condition_k_geq_rankL}
		Let $d_\rho(w,c)$ be an optimal solution to \eqref{op:uv} with $w=a_{1:k}, c=\sqrt{\sum_{i=k+1}^{n_2}a_i^2}$, where $1\leq k<n_2$ and nonzero vector $a\in\mathbb{R}^{n_2}$ is a given vector such that $a_1\geq\dots\geq a_{n_2}\geq0$. Denote $\ell=\|w\|_0$.
		Then,  $[d_\rho(w,c);0_{n_2-k}]$ is an optimal solution to \eqref{op:2} with $\tau=\rho$ if and only if
		\begin{equation}
			\label{eq:condition_k_geq_rankL}
			\rho \geq \frac{a_{\ell+1}}{\sqrt{\ell a_{\ell+1}^2 +c^2}}.
		\end{equation}
		
	\end{proposition}
	
	Proposition \ref{prop:condition_k_geq_rankL} establishes a criterion for assessing whether $[d_\rho(w,c);0_{n_2-k}]$ constitutes an optimal solution to \eqref{op:2}.
	While it is possible to verify this by checking the optimality conditions of \eqref{op:2} at $[d_\rho(w,c);0_{n_2-k}]$,  Proposition \ref{prop:condition_k_geq_rankL} offers a more direct and profound insight.
	Then, based on previous discussions, the validity of $k$ can now be determined easily, albeit at the cost of computing one additional singular value of $A$ due to the presence of $a_{\ell+1}$ in \eqref{eq:condition_k_geq_rankL}.
	Indeed, if the condition in \eqref{eq:condition_k_geq_rankL} is satisfied, then $k$ is an appropriate choice, meaning that $k \geq \mathrm{rank}(L^*)$. Otherwise, since $a_{\ell+1}/{\sqrt{\ell a_{\ell+1}^2 +c^2}}$ is  nonincreasing in terms of $\ell$,  $k$ shall be iteratively increased until \eqref{eq:condition_k_geq_rankL} is eventually satisfied.

	Now, we are able to, as promised at the beginning of this section, accelerate the solving of \eqref{op:L} based on  Theorems \ref{thm:opt_UV_general} and \ref{thm:overparameter}.
	The detailed steps are summarized in Algorithm \hyperref[algo:update_L_acc]{acc\_updateL}. As one can observe, a ``while'' loop is incorporated to ensure the validity of the estimated rank $k$. However, this also implies that numerous partial SVD may be required.
	Fortunately, we observe in our numerical experiments with Algorithm  \hyperref[algo:AltMin]{AltMin} that the rank of the iterates $\{L^i\}$ tends to  stabilize after only a few iterations. This phenomenon may be related to the partial smoothness of the objective function in \eqref{prob:srpcp} \cite{lewis2002active,daniilidis2014orthogonal,Viter2018Model}  and the finite rank-identification property of the alternating minimization algorithm, which will be further investigated in the future work.
	This observation also suggests a practical heuristic: initializing the rank estimate in each iteration based on the rank of the solution from the previous iteration.
	It is worth noting that with this heuristic, the ``while'' loop usually takes only one iteration in our numerical tests.

	\begin{algorithm}[H] \label{algo:update_L_acc}
		\caption{{\bf acc\_updateL}: An accelerated version of Step 1 in Algorithm  \hyperref[algo:AltMin]{AltMin} $(L,k)={\rm \bf acc\_updateL}(A,\rho,k,\Delta k)$}
		\vskip6pt
		\begin{algorithmic}
			\State\textbf{Initialization }  Given the data matrix $A$, parameter $\rho>0$, estimated rank $k$, and chosen updating integer $\Delta k >0$.
			
			\While{ \eqref{eq:condition_k_geq_rankL} not satisfied }
			\State\textbf{Step 1. } Increase rank estimate: $k \leftarrow k + \Delta k$.
			
			\State\textbf{Step 2. } Calculate the partial SVD of $A$ using the Matlab command {\tt svds}:
			\begin{equation}
				(H,\Sigma,W)=\mbox{\tt svds}(A,k+1).
			\end{equation}
			
			\State\textbf{Step 3. } Set $\sigma={\rm diag}(\Sigma$), $w=\sigma_{1:k}$,
			and $c=\sqrt{\|A\|_F^2-\|w\|_2^2}$.
			
			\State\textbf{Step 4. } Compute the optimal solution $d_{\rho}(w,c)$ via \eqref{eq:lowrankd}.
			\EndWhile
			\State\textbf{Step 5. } Set $U = H_{:,1:k} {\rm Diag}(\sqrt{d_\rho(w,c)}),V = W_{:,1:k} {\rm Diag}(\sqrt{d_\rho(w,c)})$.

			\State\textbf{Output: } $(UV^T,\|d_\rho(w,c)\|_0)$.
		\end{algorithmic}
	\end{algorithm}

	With all these preparations, we now provide the accelerated version of Algorithm  \hyperref[algo:AltMin]{AltMin} for efficiently solving problem \eqref{prob:srpcp}.
	The detailed steps are summarized in  Algorithm \hyperref[algo:Over]{Acc\_AltMin}.

	\begin{algorithm}[H] \label{algo:Over}
		\caption{{\bf Acc\_AltMin}: An accelerated alternating minimization method for solving \eqref{prob:srpcp}}
		\vskip6pt
		\begin{algorithmic}
			
			\State\textbf{Initialization } Given the data matrix $D\in\mathbb{R}^{n_1\times n_2}$, parameters $\lambda>0$, $\mu>0$, choose any $S^0 \in \mathbb{R}^{n_1\times n_2}$ and set the estimated rank $k^0$.
			
			\For{ $i=0,1\dots,$ }
			\State\textbf{Step 1. }
			Update
			\begin{equation*}
				S^{i+1} \in\arg\min_S \,\, \{  \lambda\|S\|_1 + \mu\|L^{i}+S-D\|_F \}.
			\end{equation*}

			\State\textbf{Step 2. } Update $L^{i+1}$ and the estimated rank $k^{i+1}$:
			\begin{equation*}
				(L^{i+1},k^{i+1}) = {\rm \bf acc\_updateL}(D-S^{i+1},1/\mu,k^i,1).
			\end{equation*}
			
			\EndFor
		\end{algorithmic}
		
	\end{algorithm}
	
	\section{Numerical experiments}
	\label{sec:num}
	In this section, we compare the performances of our algorithms, i.e., Algorithm  \hyperref[algo:AltMin]{AltMin} and its accelerated counterpart Algorithm \hyperref[algo:Over]{Acc\_AltMin}, and the ADMM used by %
	\cite{zhang2021square} for solving the SRPCP problem \eqref{prob:srpcp}.
	The experiments are conducted by running Matlab (version 9.12) on a Linux workstation (128-core, Intel Xeon Platinum 8375C @ 2.90GHz, 1024 Gigabytes of RAM).
	We use the Matlab command {\tt svds} with default settings to compute the partial SVD in Algorithm~ \hyperref[algo:update_L_acc]{acc\_updateL}.
	
	In our numerical experiments, we measure the accuracy of an approximate optimal solution $(\widehat L, \widehat S)$ for problem \eqref{prob:srpcp} by using the following relative residual:
	\begin{equation}
		\label{eq:KKT res}
		\begin{aligned}
			\eta &= \frac{\Delta_1(\widehat L, \widehat S)+\Delta_2(\widehat L, \widehat S)}{1+\|\widehat L\|_F+\|\widehat S\|_F},
		\end{aligned}	
	\end{equation}
	where
	\[
	\Delta_1(\widehat L, \widehat S)=\|\widehat L-\operatorname{prox}_{\|\cdot\|_*}(\widehat L -\mu\nabla f(\widehat L, \widehat S))\|_F,\, \quad
	\Delta_2 (\widehat L, \widehat S) =\|\widehat S-\operatorname{prox}_{\lambda\|\cdot\|_1}(\widehat S-\mu\nabla f(\widehat L,  \widehat S))\|_F\\
	\] and
	$f(L,S):=\|L+S-D\|_F$. Here, as an approximate solution, $(\widehat{L}, \widehat{S})$ is assumed implicitly close to the optimal solution set to \eqref{prob:srpcp}. Hence, by the non-overfitting assumption, we can assume that $f$ is differentiable at $(\widehat{L}, \widehat{S})$.
	Note that in our algorithms  \hyperref[algo:AltMin]{AltMin} and \hyperref[algo:Over]{Acc\_AltMin}, since the exact optimal solutions (up to machine accuracy) to the involved subproblems are computed in each iteration, the final output of both algorithms, denoted by $(L^K, S^K)$, satisfies
	\[
	L^{K}=\operatorname{prox}_{\|\cdot\|_*}\left(L^{K}-\mu\nabla f(L^{K},S^{K})\right),
	\]
	i.e., $\Delta_2(L^K, S^K) =0$.
	Let $\epsilon>0$ be a given tolerance. We terminate both our algorithms  \hyperref[algo:AltMin]{AltMin} and \hyperref[algo:Over]{Acc\_AltMin} when $\eta<\epsilon$.
	For the ADMM, we terminate it when its default stopping criterion described in Section~3 of \cite{zhang2021square} reaches the tolerance $\epsilon$.
	All the tested algorithms will also be stopped when they reach the maximum computational time of 5 hours.
	For all the experiments in this section, the initial points are set as zero matrices.
	We also compare the quality of approximate solutions by measuring the objective function values, where a lower objective function value indicates a better solution.
	
	\subsection{Synthetic data}\label{sec:syn data}
	
	First we compare Algorithm  \hyperref[algo:AltMin]{AltMin} and the ADMM under various noise levels and dimensions using synthetic data.
	We adopt the procedure in Section~4 of \cite{zhang2021square} to generate the data.
	We set $n_1=n_2=n$ in the observation model \eqref{obser-model}.
	The ground truth low-rank matrix $L_0\in\mathbb{R}^{n\times n}$ is constructed as $L_0=XY^T$, where $X\in\mathbb{R}^{n\times r}$ and $Y\in\mathbb{R}^{n\times r}$ have i.i.d. entries drawn from $\mathcal{N}(0,1/n)$.
	Then, we generate the ground truth sparse matrix $S_0\in\{-1,0,1\}^{n\times n}$ having a support set of size $s$ chosen uniformly at random and independent random signs, and the random noise matrix $Z_0$ with  entrywise i.i.d. $\mathcal{N}(0,\sigma^2)$.
	In the test, we take $\lambda = 1/\sqrt{n}$ and $\mu = \sqrt{n/2}$, and set the matrix dimensions $n$, the rank of $L_0$, and the sparsity of $S_0$ by
	\[
	n\in\{10^3,2\times10^3,5\times10^3,10^4\}, \ s=\|S_0\|_0=0.05n^2, \ r={\rm rank}(L_0)=20.
	\]
	For the noise level $\sigma$, we consider $\sigma \in \{10^{-1},10^{-2},10^{-3},10^{-4}\}$.
	In addition to evaluating the objective function value, we also report the relative recovery errors for the low-rank and sparse components, defined as follows:
	\begin{equation*}
		\eta_L:=\frac{\|\widehat{L}-L_0\|_F}{1+\|L_0\|_F},\quad\eta_S:=\frac{\|\widehat{L}-S_0\|_F}{1+\|S_0\|_F},
	\end{equation*}
	where $(\widehat{L},\widehat{S})$
	denote the approximated solutions returned by the tested algorithms.
	We set stopping tolerance $\epsilon=10^{-6}$ in this experiment.

	\begin{table}
		\centering
		\caption{Comparisons between \hyperref[algo:AltMin]{AltMin} and the ADMM on synthetic data.
			The computational time is in the format of ``hours:minutes:seconds".
			The ``$\eta_S$'' and ``$\eta_L$'' columns stand for the relative recovery error of $S$ and $L$, respectively.}
		\label{table:1}
		\begin{tabular}{ccccccccc}
			\toprule
			& \multicolumn{4}{c}{AltMin} & \multicolumn{4}{c}{ADMM} \\
			\cmidrule(lr){2-5} \cmidrule(lr){6-9}
			$\sigma$ & time & $\eta_S$ & $\eta_L$ & obj & time & $\eta_S$ & $\eta_L$ & obj \\
			\midrule
			\multicolumn{9}{c}{$n=10^3$} \\
			\midrule
			$10^{-1}$ & {\bf 3} & 3.87e0 & 2.34e0 & {\bf 2133.43} & 8 & 3.87e0 & 2.34e0 & 2133.46 \\
			$10^{-2}$ & {\bf 4} & 6.06e-1 & 3.59e-1 & {\bf 350.04} & 9 & 6.06e-1 & 3.59e-1 & 350.07 \\
			$10^{-3}$ & {\bf 5} & {\bf 7.19e-2} & 3.68e-2 & {\bf 196.27} & 26 & 7.20e-2 & 3.68e-2 & 196.29 \\
			$10^{-4}$ & {\bf 10} & 7.38e-3 & 3.70e-3 & {\bf 179.47} & 38 & {\bf 7.32e-3} & 3.70e-3 & 179.49 \\
			\midrule
			\multicolumn{9}{c}{$n=2\times 10^3$} \\
			\midrule
			$10^{-1}$ & {\bf 12} & {\bf 7.65e0} & 2.41e0 & {\bf 6028.95} & 50 & 7.66e0 & 2.41e0 & 6029.02 \\
			$10^{-2}$ & {\bf 14} & 1.01e0 & 3.63e-1 & {\bf 966.83} & 49 & 1.01e0 & 3.63e-1 & 966.86 \\
			$10^{-3}$ & {\bf 18} & 1.18e-1 & 3.68e-2 & {\bf 517.81} & 1:35 & {\bf 1.17e-1} & 3.68e-2 & 517.88 \\
			$10^{-4}$ & {\bf 29} & 1.20e-2 & 3.69e-3 & {\bf 472.57} & 1:51 & {\bf 1.18e-2} & 3.69e-3 & 472.63 \\
			\midrule
			\multicolumn{9}{c}{$n=5\times10^3$} \\
			\midrule
			$10^{-1}$ & {\bf 1:55} & 1.90e1 & 2.46e0 & {\bf 23831.70} & 12:35 & 1.90e1 & 2.46e0 & 23831.88 \\
			$10^{-2}$ & {\bf 2:12} & 2.29e0 & 3.66e-1 & {\bf 3786.80} & 5:19 & 2.29e0 & 3.66e-1 & 3787.12 \\
			$10^{-3}$ & {\bf 3:03} & {\bf 2.14e-1} & {\bf 2.14e-1} & {\bf 270.82} & 17:21 & 2.15e-1 & 2.16e-1 & 270.93 \\
			$10^{-4}$ & {\bf 3:15} & 2.54e-2 & 3.96e-3 & {\bf 1808.80} & 20:04 & {\bf 2.49e-2} & 3.96e-3 & 1809.01 \\
			\midrule
			\multicolumn{9}{c}{$n=10^4$} \\
			\midrule
			$10^{-1}$ & {\bf 9:09} & 3.80e1 & 2.47e0 & {\bf 67391.20} & 43:57 & 3.81e1 & 2.47e0 & 67392.10 \\
			$10^{-2}$ & {\bf 10:38} & 4.44e0 & 3.68e-1 & {\bf 10695.74} & 39:38 & 4.44e0 & 3.68e-1 & 10696.07 \\
			$10^{-3}$ & {\bf 11:27} & 4.70e-1 & 3.69e-2 & {\bf 5586.82} & 56:00 & {\bf 4.66e-1} & 3.69e-2 & 5587.62 \\
			$10^{-4}$ & {\bf 14:19} & 4.75e-2 & 3.69e-2 & {\bf 5077.41} & 1:18:04 & {\bf 4.45e-2} & 3.69e-2 & 5078.12 \\
			\bottomrule
		\end{tabular}
	\end{table}
	
	In Table \ref{table:1}, we present numerical results of Algorithm  \hyperref[algo:AltMin]{AltMin} and the ADMM for solving various instances of problem \eqref{prob:srpcp}.
	The results clearly demonstrate the high efficiency of Algorithm  \hyperref[algo:AltMin]{AltMin}.
	Despite a slight increase in recovery errors compared to the ADMM, our algorithm achieves a lower optimal value and significantly reduces computational time.
	It is evident that \hyperref[algo:AltMin]{AltMin} is more than twice as fast as the ADMM for all tests, and can even be up to $6$ times faster (when $n=5\times 10^3,\sigma=10^{-4}$).
	
	Next, we evaluate the efficiency of the accelerated method \hyperref[algo:Over]{Acc\_AltMin} under the low-rank setting.
	Here, we primarily focus on assessing the algorithm's performance in high-dimensional scenarios.
	For this test, the same data generation procedure as in the previous test is adopted, but adjust several parameters, including the dimension of $L_0$, its rank $r$, and the choice of the parameter $\mu$.
	Specifically, we test a range of dimensions $n\in\{10^4,2\times 10^4,5\times 10^4\}$ and $r={\rm rank}(L_0)=20~ \text{or}~50$.
	Additionally, the parameter $\mu$ is adjusted dynamically to control the rank of the approximated solutions returned by the tested algorithms.
	The selection criterion for $\mu$ is to make the rank of the final solution $\widehat{L}$ as close as possible to the rank of $L_0$.
	Note that from the discussions in Section~\ref{sec:acc}, the iterates $\{(L^i, S^i)\}$ generated by  \hyperref[algo:AltMin]{AltMin} and \hyperref[algo:Over]{Acc\_AltMin} remain identical at each iteration.
	Consequently, the objective function values and recovery errors are identical across iterations of the two algorithms.
	Thus, we focus exclusively on their computational runtimes.
	Since the most time-consuming operations in both algorithms are the (partial) singular value decompositions, we will also report the computational time of these operations.
	\begin{longtable}{ccccccccc}
		\caption{Comparisons between algorithms \hyperref[algo:AltMin]{AltMin}, \hyperref[algo:Over]{Acc\_AltMin} and the ADMM on synthetic data under low-rank settings. The computational time is in the format of ``hours:minutes:seconds". ``t" stands for time out, e.g., over $5$ hours.
			The ``svd'' columns stand for the computational time of SVD(s) operations. The objective function values of Algorithms \hyperref[algo:AltMin]{AltMin} and \hyperref[algo:Over]{Acc\_AltMin} are identical, thus we only present them in the ninth column.}\label{table:2}\\
		
		\toprule
		& & \multicolumn{2}{c}{AltMin } & \multicolumn{2}{c}{ADMM} & \multicolumn{3}{c}{Acc\_AltMin} \\
		\cmidrule(lr){3-4} \cmidrule(lr){5-6} \cmidrule(lr){7-9}
		$\sigma$ & $\mu$ & time & svd  & time & obj & time & svd & obj \\
		\midrule
		\endfirsthead
		
		\toprule
		& & \multicolumn{2}{c}{AltMin } & \multicolumn{2}{c}{ADMM} & \multicolumn{3}{c}{Acc\_AltMin} \\
		\cmidrule(lr){3-4} \cmidrule(lr){5-6} \cmidrule(lr){7-9}
		$\sigma$ & $\mu$ & time & svd  & time & obj & time & svd & obj \\
		\midrule
		\endhead
		
		\bottomrule
		\endfoot
		
		\bottomrule
		\endlastfoot
		
		\multicolumn{9}{c}{$n = 10^4$,\hspace{3em}$r = 20$} \\
		\midrule
		$10^{-1}$ & $50.69$ & 6:04 & 5:22 & 50:58 & 50936.70 & {\bf 3:39} & 3:04 & {\bf 50935.81} \\
		$10^{-2}$ & $50.69$ & 5:58 & 5:18  & 33:57 & 8756.65 & {\bf 4:07} & 3:28 & {\bf 8755.68} \\
		$10^{-3}$ & $49.80$ & 5:57 & 5:13  & 1:16:10 & 5381.00 & {\bf 2:01} & 1:23 & {\bf 5380.52} \\
		$10^{-4}$ & $48.77$ & 4:29 & 3:52  & 1:11:15 & 5057.76 & {\bf 1:13} & 36 & {\bf 5056.76} \\
		\midrule
		\multicolumn{9}{c}{$n = 10^4$,\hspace{3em}$r = 50$} \\
		\midrule
		$10^{-1}$ & $55.24$ & 5:40 & 4:58  & 1:11:15 & 51460.03 & {\bf 3:37} & 3:04 & {\bf 51459.06} \\
		$10^{-2}$ & $51.24$ & 5:25 & 4:48  & 35:07 & 8828.25 & {\bf 3:34} & 2:58 & {\bf 8827.45} \\
		$10^{-3}$ & $50.51$ & 3:56 & 2:54  & 35:07 & 5414.48 & {\bf 1:19} & 17 & 5414.48 \\
		$10^{-4}$ & $50.15$ & 5:07 & 4:25  & 1:14:36 & 5085.50 & {\bf 1:33} & 50 & {\bf 5084.52} \\
		\midrule
		\multicolumn{9}{c}{$n = 2\times 10^4$,\hspace{3em}$r = 20$} \\
		\midrule
		$10^{-1}$ & 71.43 & 18:23 & 16:12  & 3:27:01 & 143571.76 & {\bf 9:45} & 7:35 & {\bf 143569.94} \\
		$10^{-2}$ & $71.17$ & 18:34 & 16:29  & 1:28:29 & 24635.24 & {\bf 9:59} & 7:58 & 24635.24 \\
		$10^{-3}$ & $70.42$ & 18:21 & 16:23 & 3:27:52 & 15184.23 & {\bf 6:33} & 4:31 & 15184.23 \\
		$10^{-4}$ & $69.44$ & 21:07 & 19:14  & t & 14279.84 & {\bf 4:47} & 2:24 & {\bf 14266.22} \\
		\midrule
		\multicolumn{9}{c}{$n = 2\times 10^4$,\hspace{3em}$r = 50$} \\
		\midrule
		$10^{-1}$ & $71.94$ & 18:37 & 16:36  & 3:38:54 & 144564.12 & {\bf 10:46} & 8:40 & {\bf 144562.35} \\
		$10^{-2}$ & $71.94$ & 18:48 & 16:45 & 1:26:10 & 24823.84 & {\bf 10:39} & 8:29 & {\bf 24820.98} \\
		$10^{-3}$ & $74.57$ & 18:35 & 16:33  & 3:25:01 & 15223.16 & {\bf 5:46} & 3:38 & 15223.16 \\
		$10^{-4}$ & $70.92$ & 24:41 & 22:10  & t & 14306.24 & {\bf 6:59} & 4:20 & {\bf 14297.59} \\
		\midrule
		\multicolumn{9}{c}{$n = 5\times 10^4$,\hspace{3em}$r = 20$} \\
		\midrule
		$10^{-1}$ & $112.54$ & 3:36:13 & 3:20:23  & t & 596786.23 & {\bf 1:22:40} & 1:05:39 & {\bf 565599.49} \\
		$10^{-2}$ & $112.54$ & 3:57:53 & 3:47:34 & t & 98025.42 & {\bf 1:19:28} & 1:03:33 & {\bf 97361.28} \\
		$10^{-3}$ & $111.90$ & 4:03:54 & 3:54:18  & t & 60035.14 & {\bf 1:04:55} & 48:09 & {\bf 60021.65} \\
		$10^{-4}$ & $111.74$ & 3:33:16 & 3:18:43  & t & 56337.21 & {\bf 50:18} & 33:55 & {\bf 56334.35} \\
		\midrule
		\multicolumn{9}{c}{$n = 5\times 10^4$,\hspace{3em}$r = 50$} \\
		\midrule
		$10^{-1}$ & $112.78$ & 3:32:38 & 3:18:15 & t & 579642.23 & {\bf 1:19:19} & 1:05:12 & {\bf 566749.13} \\
		$10^{-2}$ & $112.78$ & 4:17:32 & 4:08:45 & t & 98123.54 & {\bf 1:22:04} & 1:06:20 & {\bf 97511.22} \\
		$10^{-3}$ & $112.14$ & 3:54:23 & 3:44:57 & t & 60074.54 & {\bf 58:16} & 42:50 & {\bf 60054.16} \\
		$10^{-4}$ & $111.90$ & 4:33:07 & 4:22:48 & t & 56372.14 & {\bf 53:47} & 36:24 & {\bf 56360.05}
	\end{longtable}

	Table \ref{table:2} illustrates that the decomposition approach significantly enhances the efficiency of Algorithm  \hyperref[algo:AltMin]{AltMin}.
	As one can observe, for all the tests, \hyperref[algo:Over]{Acc\_AltMin} is faster than both \hyperref[algo:AltMin]{AltMin} and the ADMM. For the largest problem with $n = 5\times 10^4$, $r = 50$ and $\sigma = 10^{-4}$, \hyperref[algo:Over]{Acc\_AltMin} returns an approximate solution of desired accuracy within 54 minutes and is at least 5 times faster than \hyperref[algo:AltMin]{AltMin}. Meanwhile, the ADMM can not obtain a solution with a reasonable accuracy within 5 hours.
	For the smaller instances with $n \in \{10^4, 2\times 10^4\}$, \hyperref[algo:Over]{Acc\_AltMin} can be 10 to 50 times faster than the ADMM. From the second and third columns of Table \ref{table:2}, we observe that the efficiency gain of \hyperref[algo:Over]{Acc\_AltMin} is largely due to the dynamic rank estimations and efficient partial SVD updates. This observation confirms our theoretical findings in Proposition~\ref{proposition:solver_UV} and Theorem~\ref{thm:opt_UV_general}.
	We can also observe that the advantage of \hyperref[algo:Over]{Acc\_AltMin} decreases as the noise level increases. We hypothesize that this is due to the distortion of the singular value structures caused by higher noise, which in turn affects the performance of the {\tt svds} algorithm.

	\subsection{Real data from dark raw videos}
	
	In many applications, video frames can be modeled as the sum of a low-rank matrix $L_0$, a sparse outlier $S_0$, and a noise matrix $Z_0$.
	The low-rank matrix $L_0$ can naturally model the background variations, and the sparse outlier matrix $S_0$ can model the foreground objects, such as cars or individuals.
	One important task in video processing is to denoise the video and decompose it into background and foreground components.
	
	In this subsection, we conduct experiments on the real dark raw video datasets.
	All the tested dark raw video datasets are sourced from the collection of \cite{chen2019seeing}, which features recordings captured under low-light conditions, resulting in an extremely low signal-to-noise ratio (SNR)—potentially negative when expressed in dB.
	Each video consists of approximately $ 110$ frames with a resolution of $ 3672 \times 5496$ pixels in the collection.
	As a preprocessing step, we convert the RGB videos to grayscale using the {\tt rgb2gray} command in Matlab, which also reduces the size of each frame to $918\times 1374$.
	Subsequently, we stack each frame of a video as a column of the observation matrix $D\in\mathbb{R}^{n_1\times n_2}$, with $n_1=918\times 1374, n_2\in\{111,112,113,114\}.$
	In the tests, we decompose the observation matrix $D$ by solving problem \eqref{prob:srpcp} with $(\lambda,\mu)=(1/\sqrt{n_1},\sqrt{n_2/2})$.
	The approximated solution pairs $(\widehat{L},\widehat{S})$ obtained from the tested algorithms yield the following decomposition: the low-rank term $\widehat{L}$  recovers the static background, while the sparse term $\widehat{S}$ captures the moving foreground objects.
	The residual noise term is consequently defined as $\widehat{Z}=D-\widehat{S}-\widehat{L}$.
	Similar to the case in the previous subsection, the  stopping tolerance of all the tested algorithms is set to $\epsilon=10^{-5}$.
	
	To illustrate the effectiveness of our Algorithm \hyperref[algo:AltMin]{AltMin} in denoising and separating the background and foreground from original data, we  show the 30-th, 60-th, and 90-th frames for the {\it windmill} video dataset in Figure~\ref{fig:DRV-frame-5}, using the approximated solution pairs returned by \hyperref[algo:AltMin]{AltMin} .
	For better presentation, we apply the {\tt imadjust} command in Matlab as a post-processing step to enhance the contrast of these images.
	As can be observed, despite the extremely low SNR, our method, with returned approximated solution pairs $(\widehat L, \widehat S)$, effectively separates the background ($\widehat{L}$, the third column) and the moving foreground ($\widehat{S}$, the fourth column).
	The reconstructed sum $\widehat{L}+\widehat{S}$ (the second column)
	not only demonstrates effective noise removal but also preserves the video's essential content compared to the original frame $D$ (the first column).	Meanwhile, the presence of the residual noise $\widehat Z$  justifies the legitimacy of the non-overfitting assumption assumed in Theorem~\ref{thm:convergence}.
	
	\begin{figure}[!htbp]
		\centering
		\includegraphics[width=0.18\textwidth]{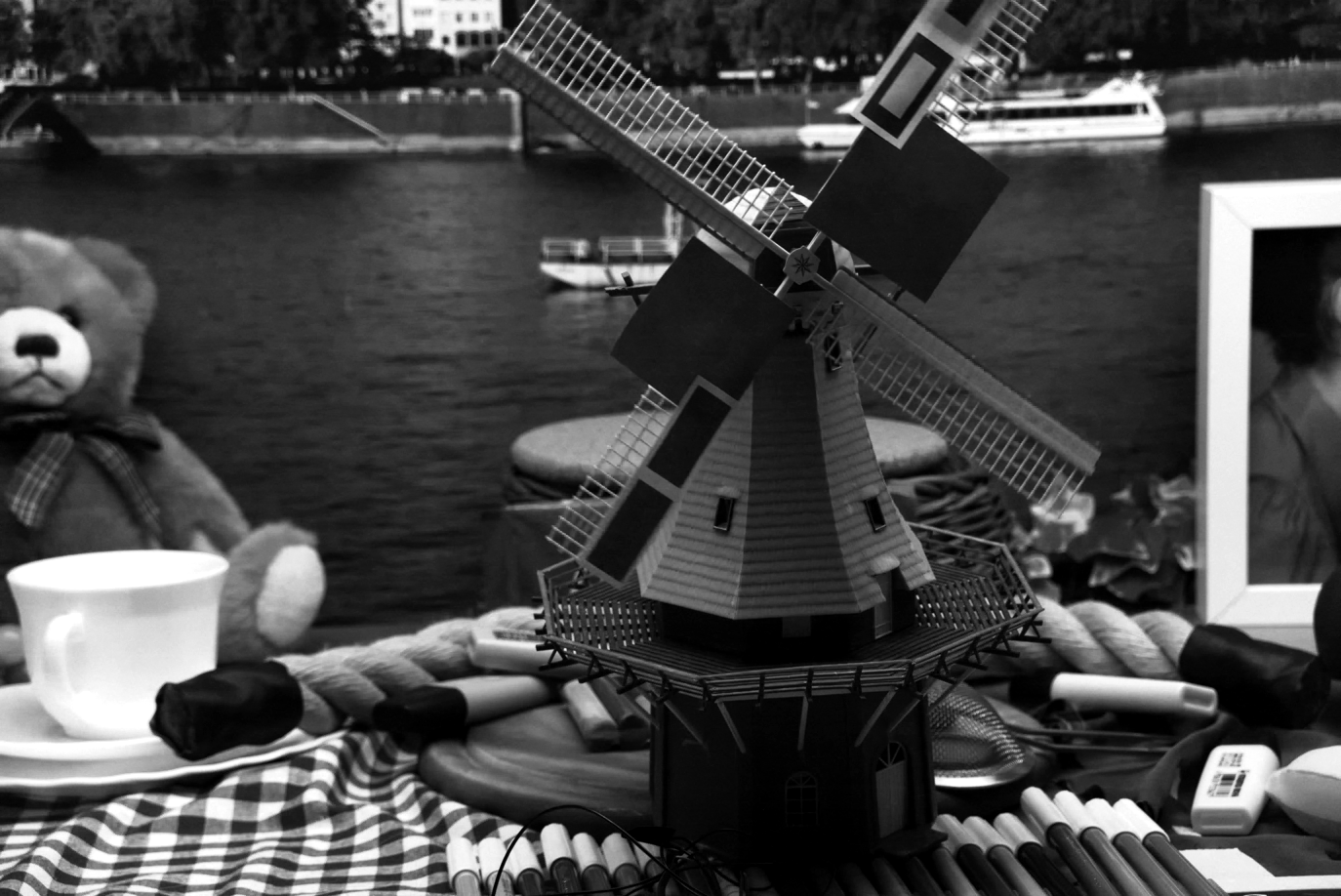}
		\includegraphics[width=0.18\textwidth]{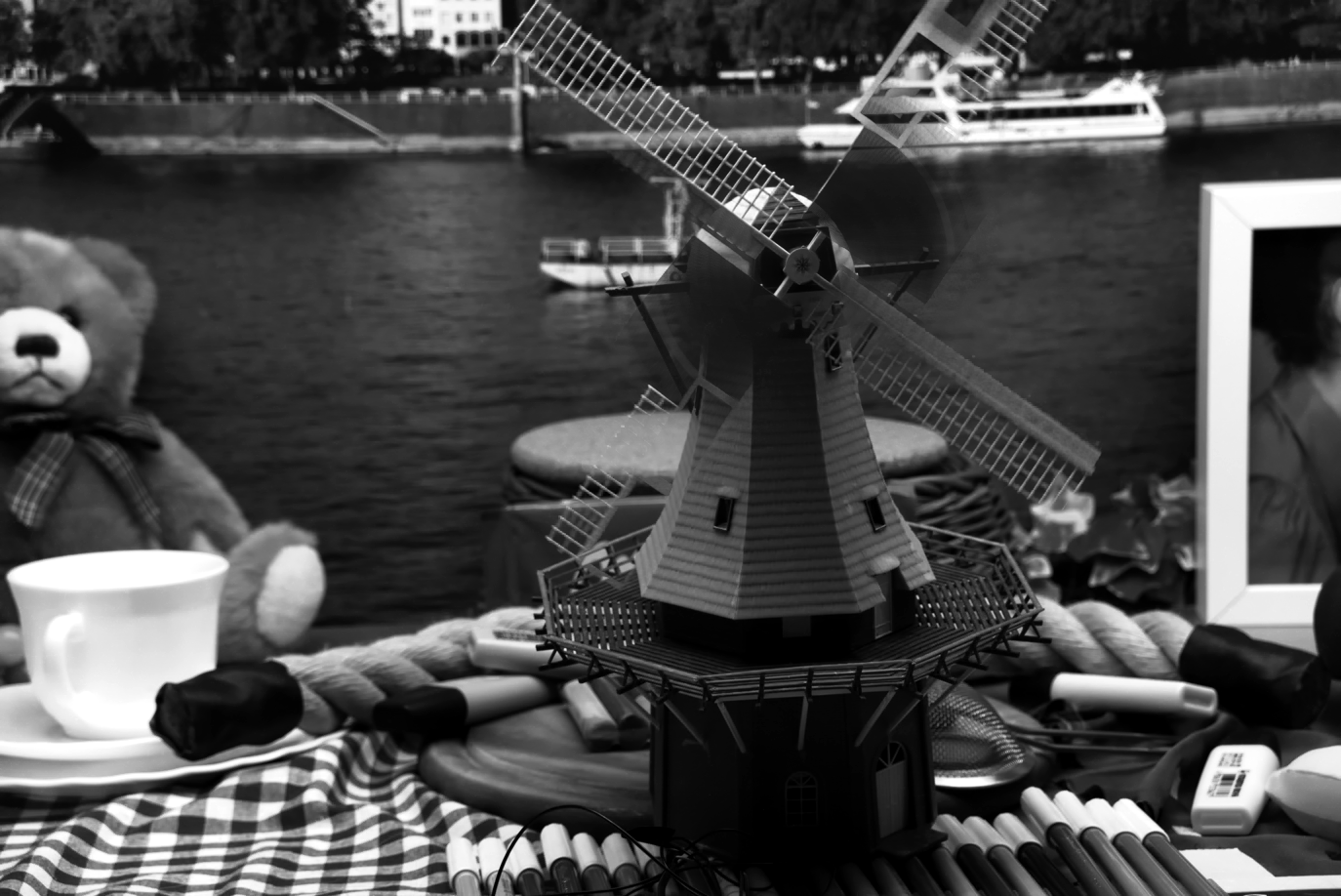}
		\includegraphics[width=0.18\textwidth]{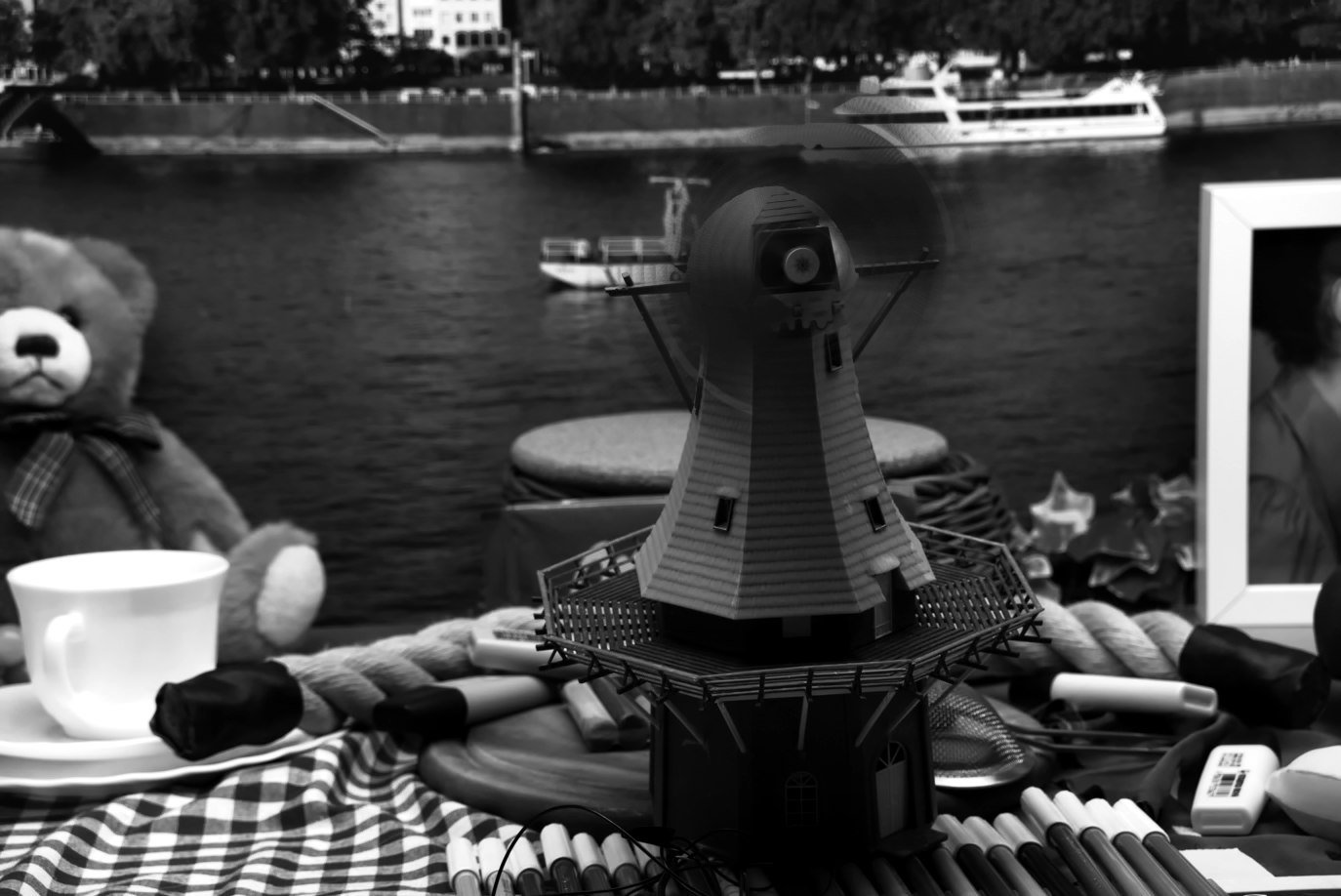}
		\includegraphics[width=0.18\textwidth]{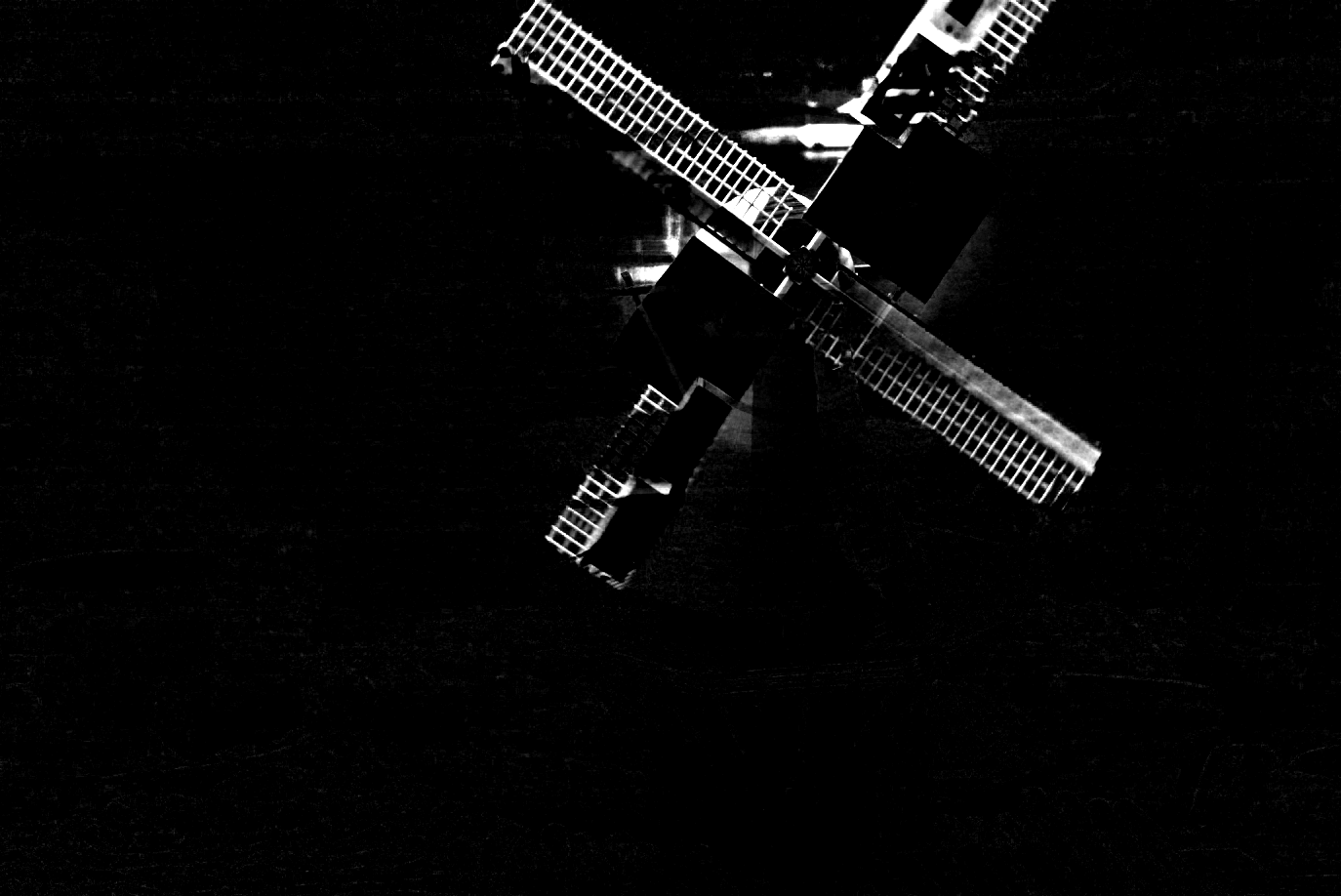}
		\includegraphics[width=0.18\textwidth]{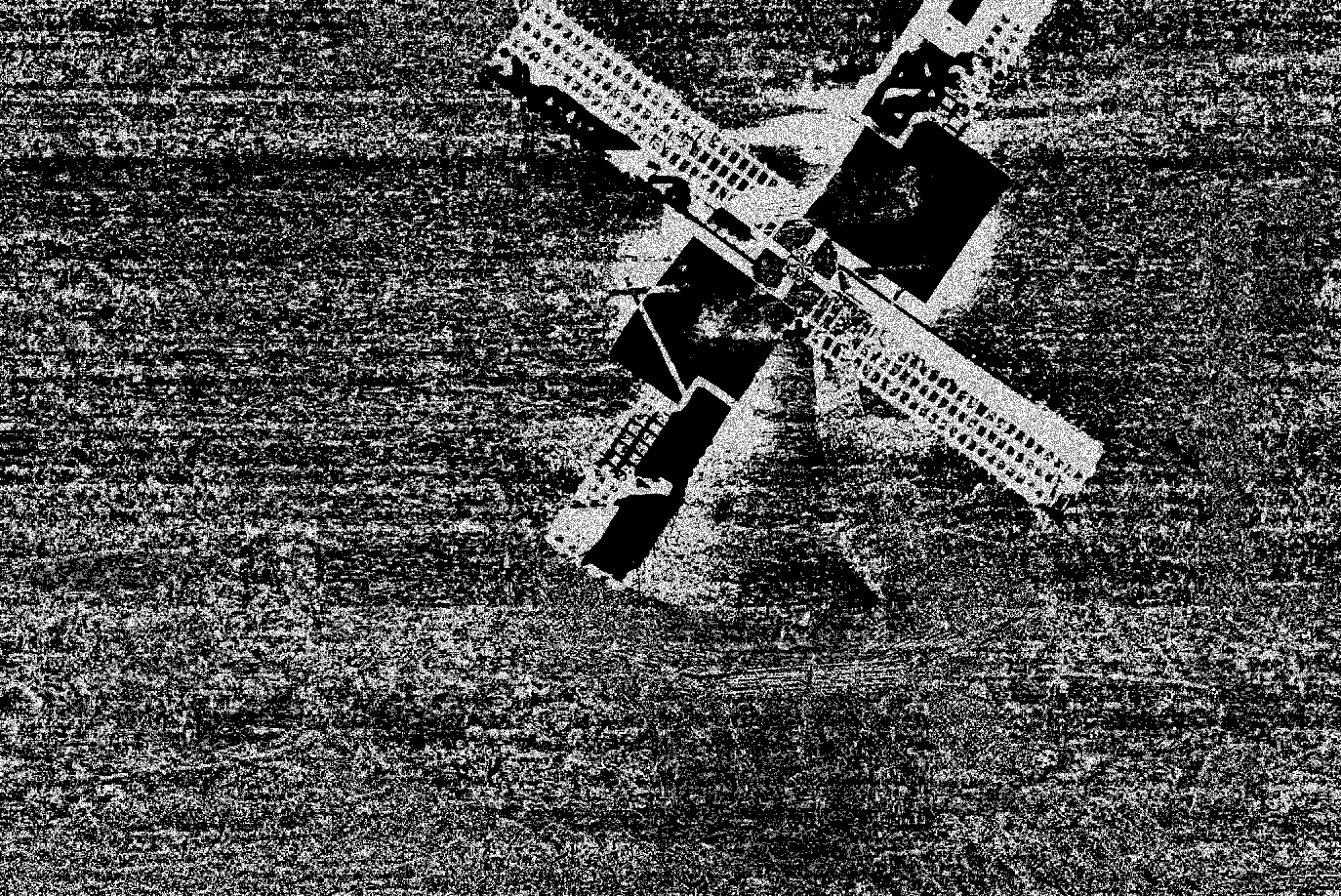}
		\vspace{10pt} \\
		\includegraphics[width=0.18\textwidth]{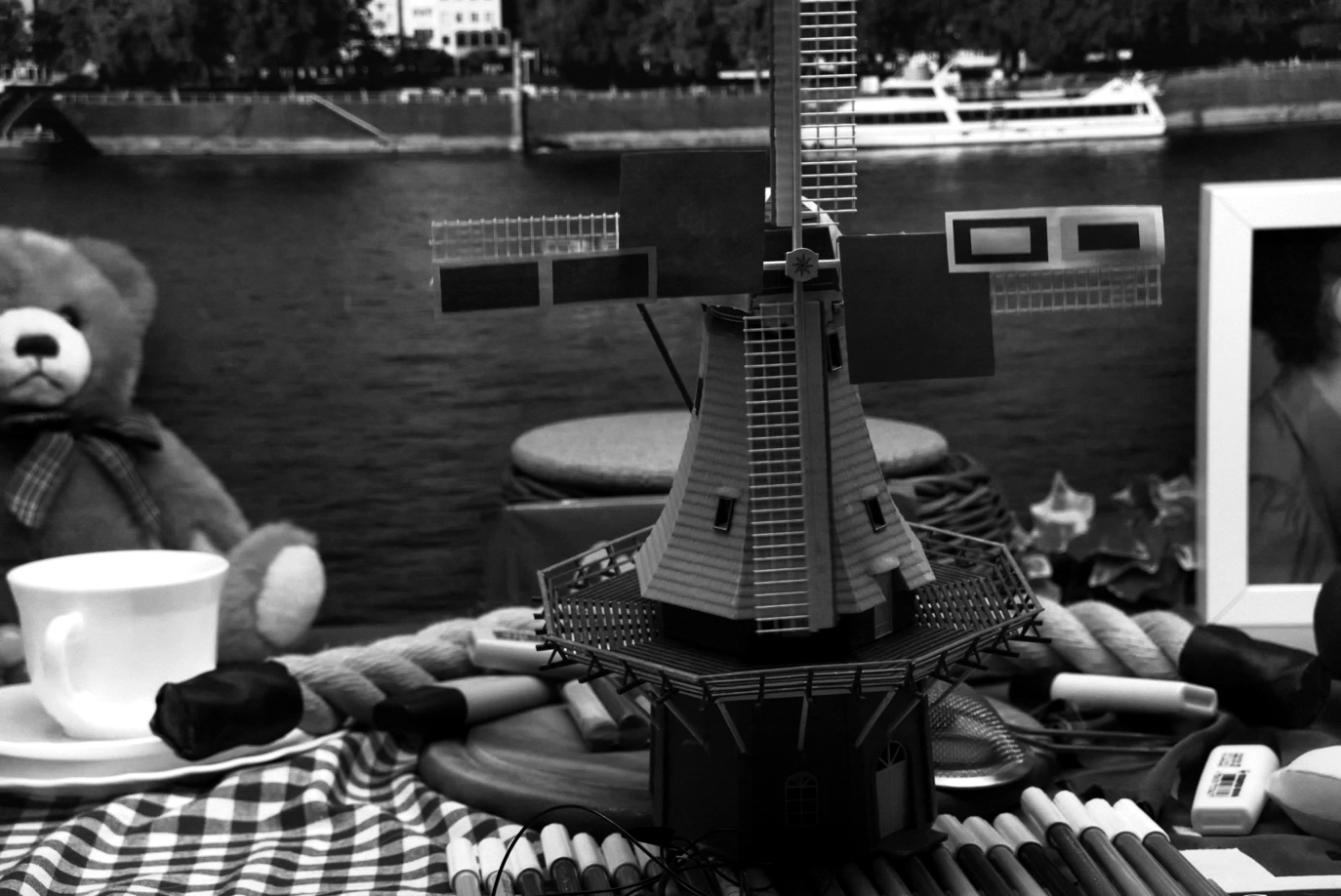}
		\includegraphics[width=0.18\textwidth]{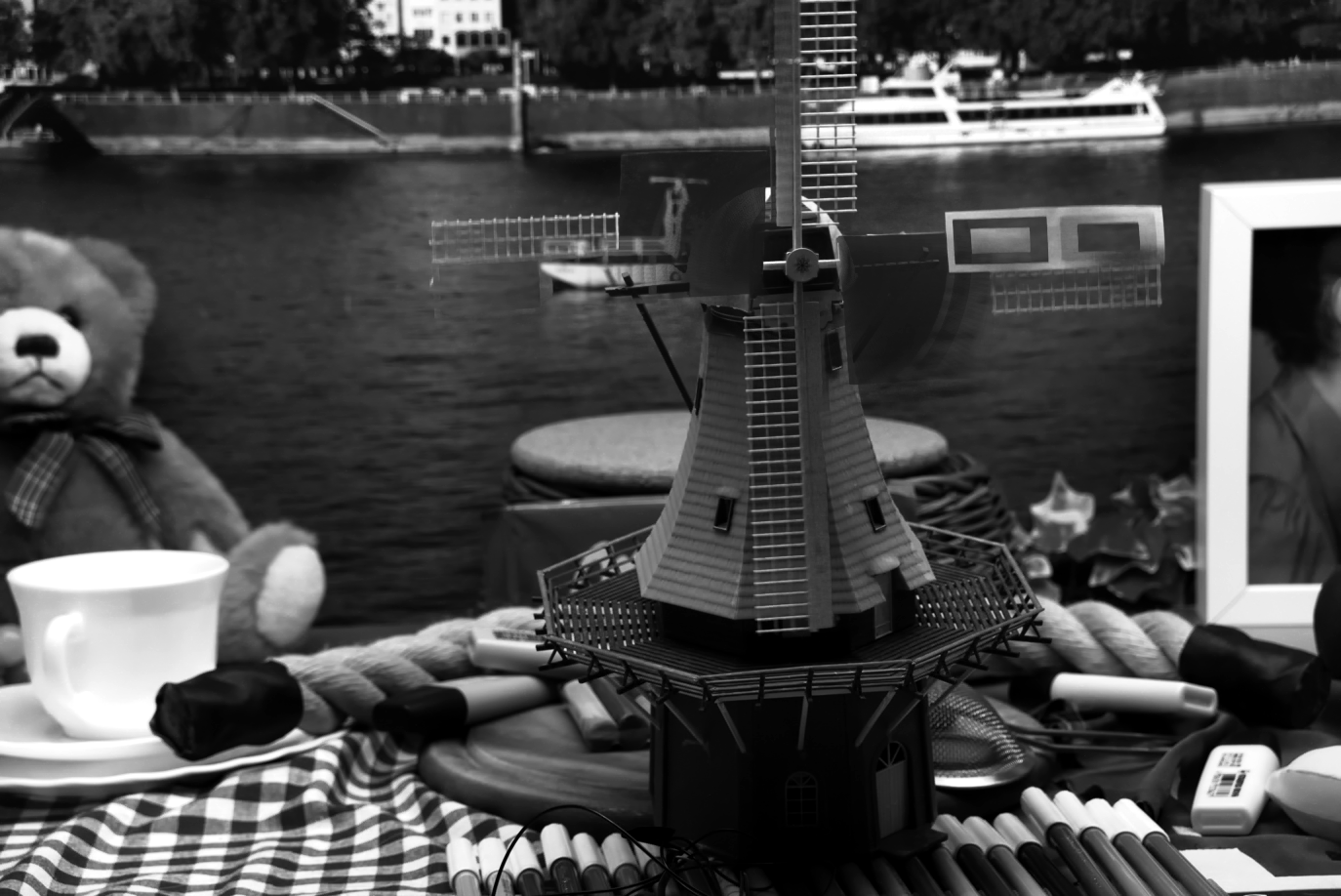}
		\includegraphics[width=0.18\textwidth]{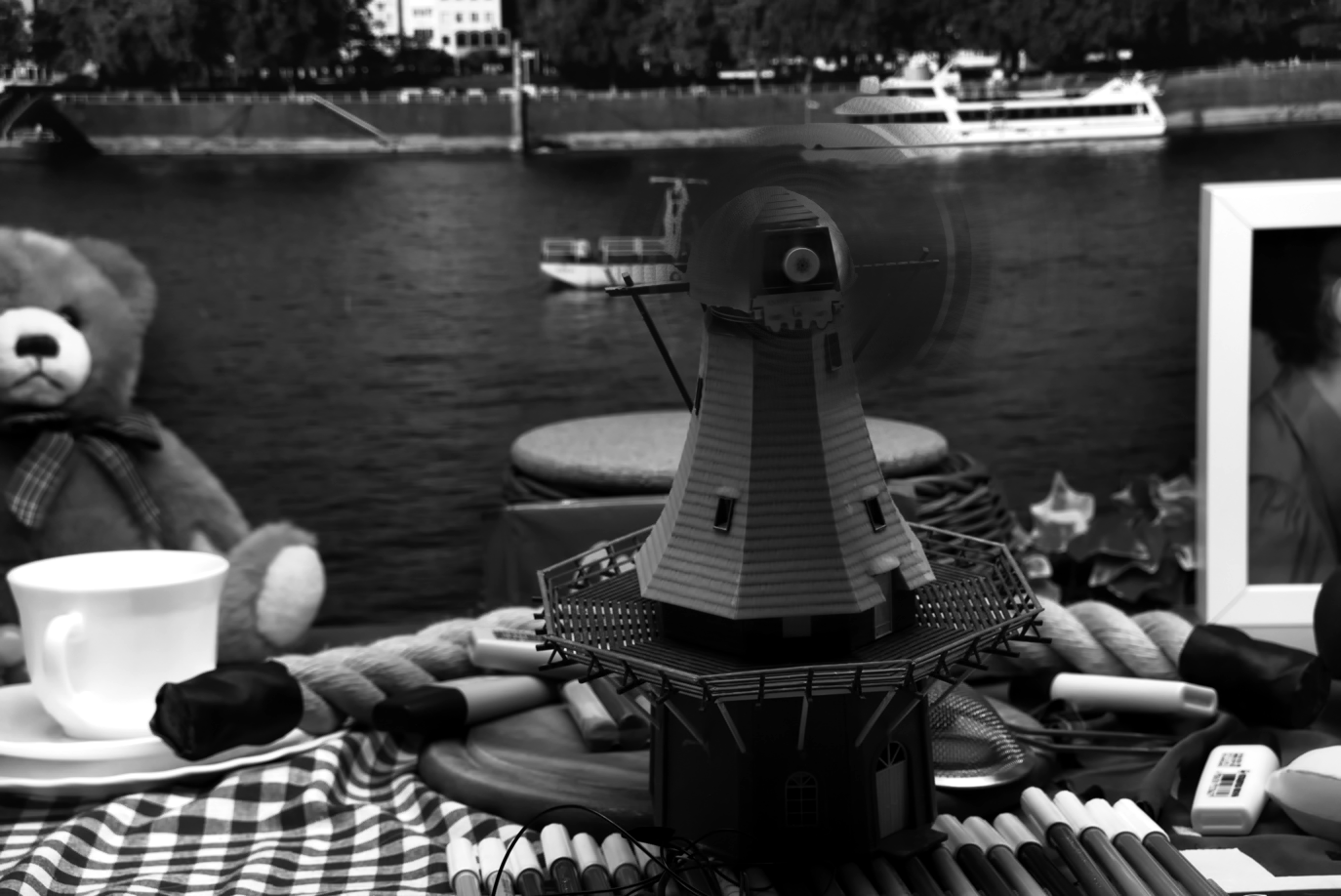}
		\includegraphics[width=0.18\textwidth]{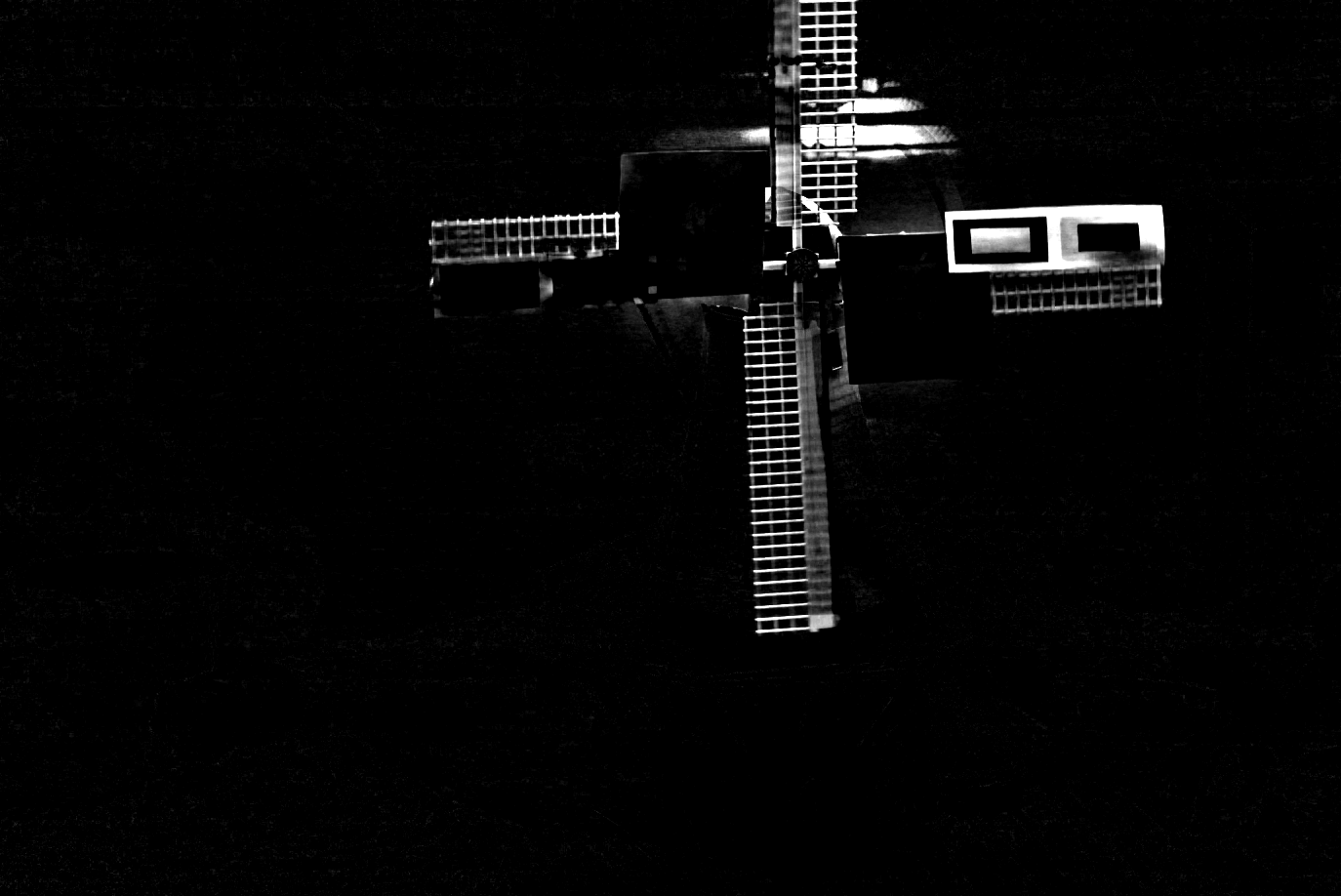}
		\includegraphics[width=0.18\textwidth]{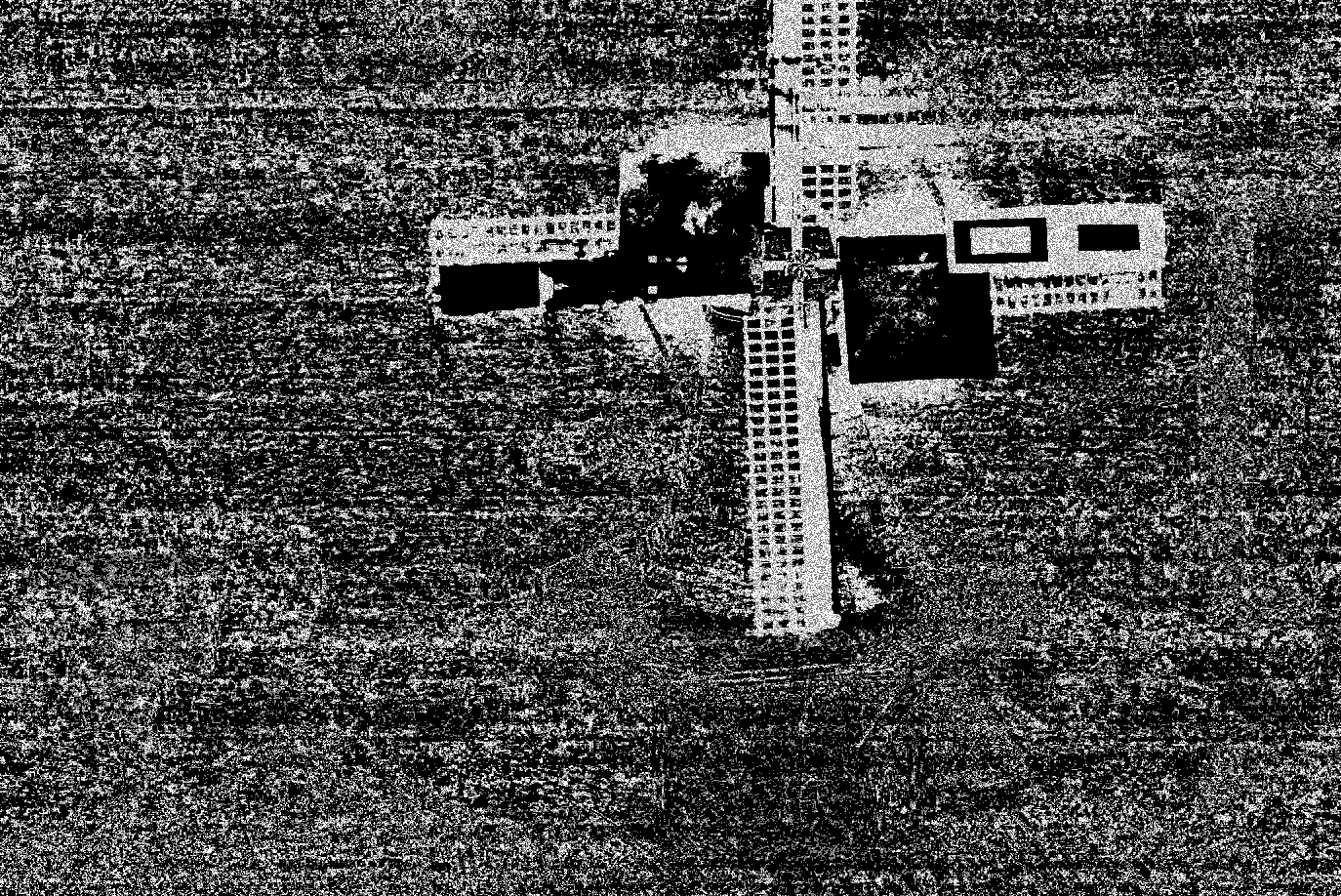}
		\\ \vspace{10pt}
		\includegraphics[width=0.18\textwidth]{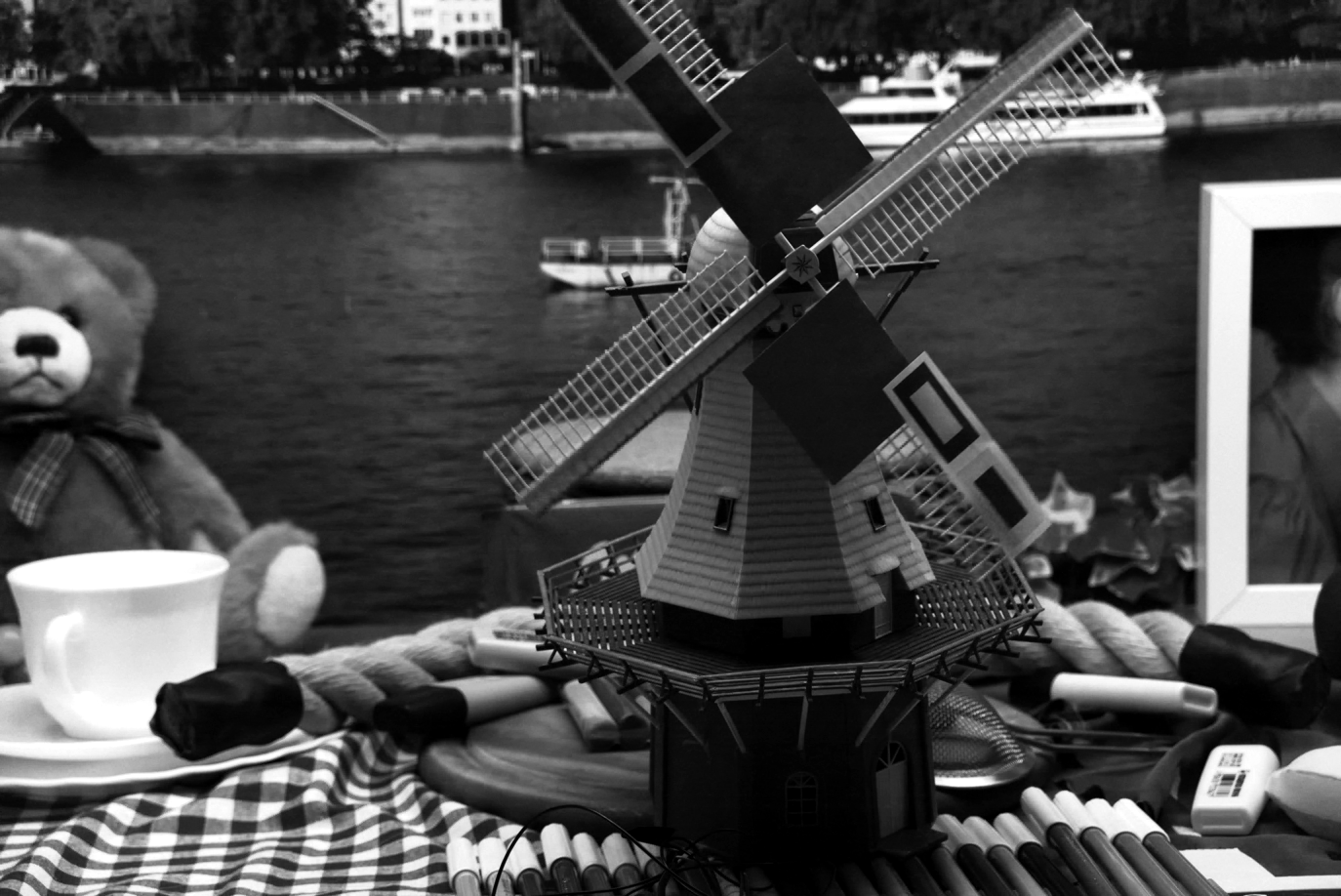}
		\includegraphics[width=0.18\textwidth]{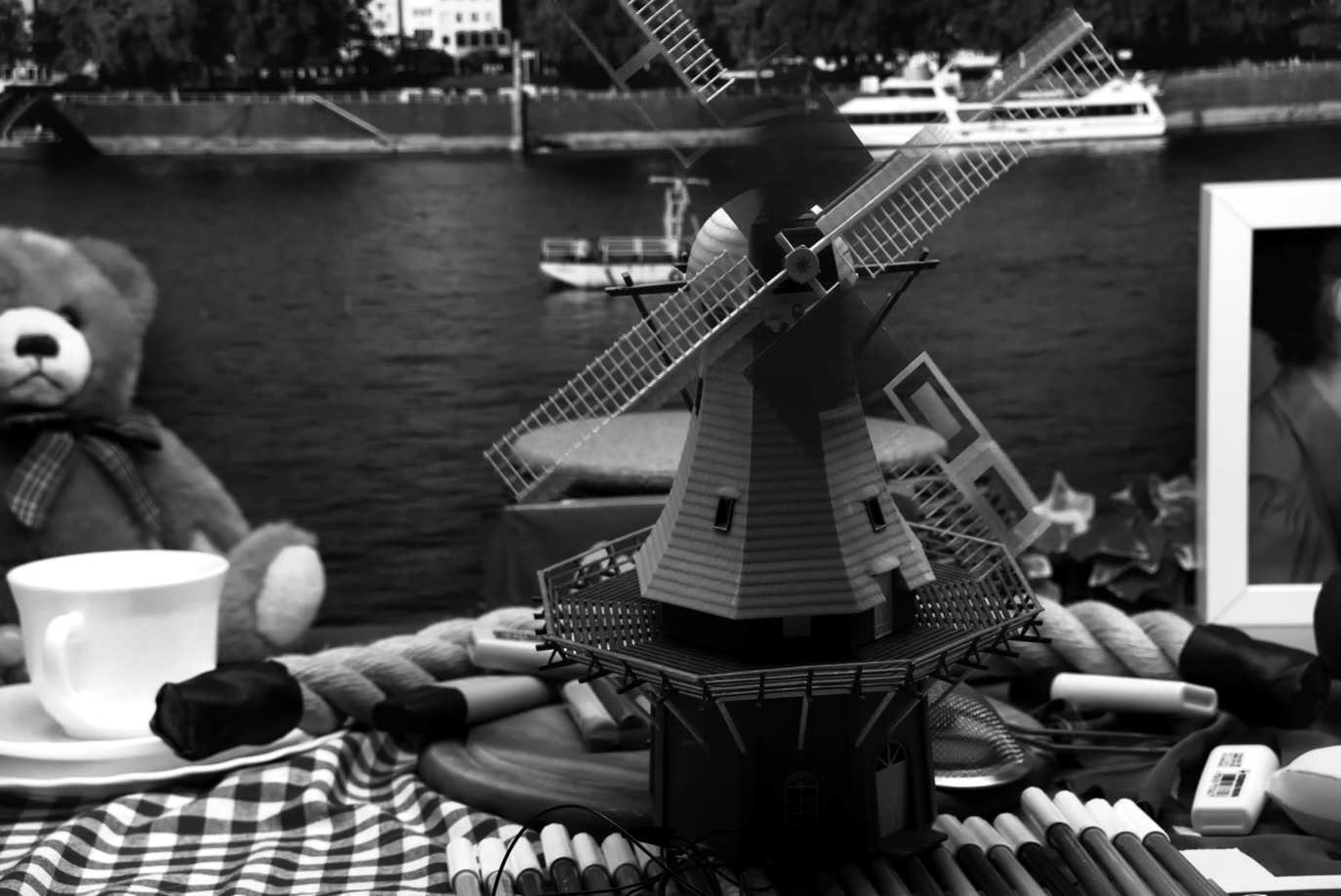}
		\includegraphics[width=0.18\textwidth]{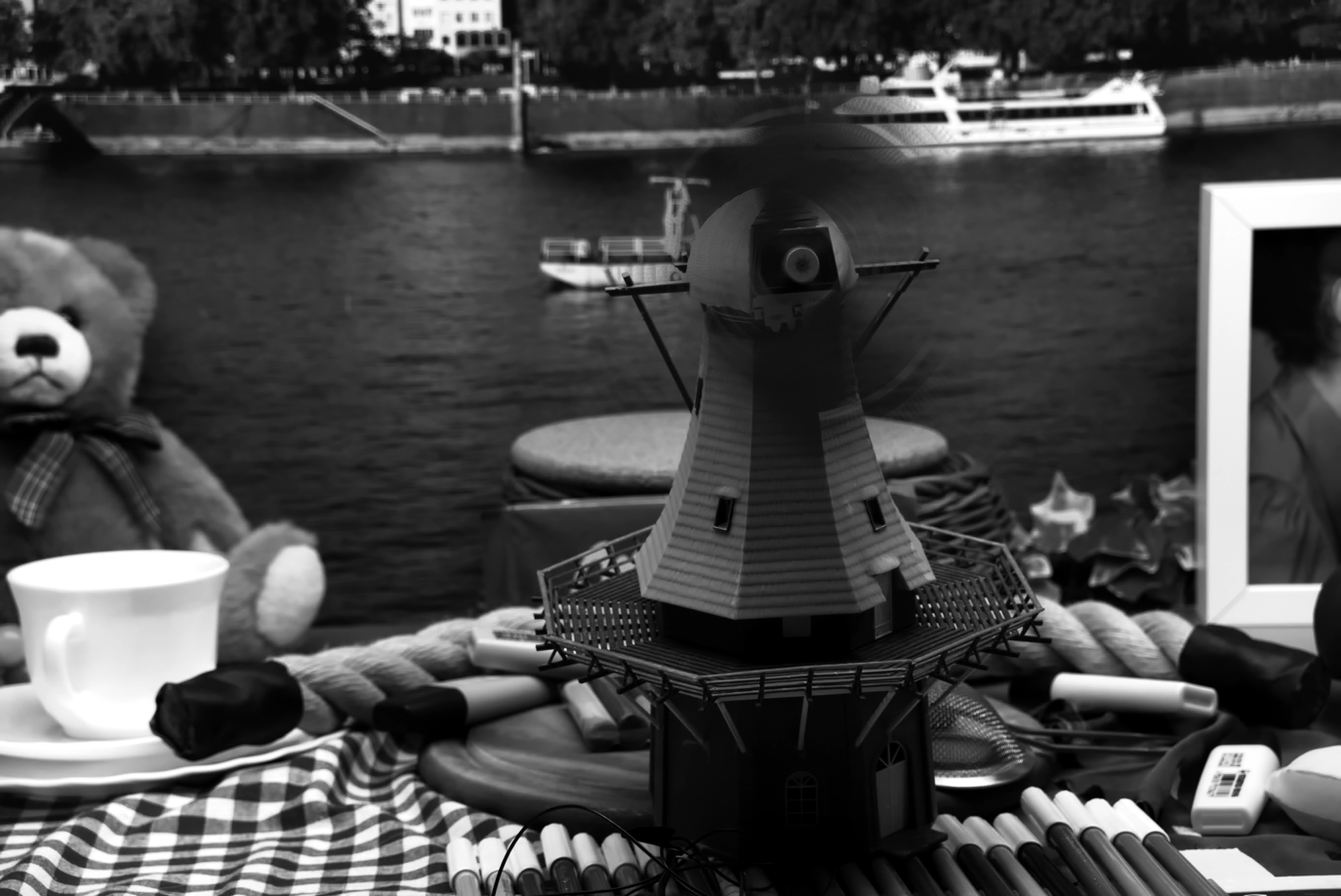}
		\includegraphics[width=0.18\textwidth]{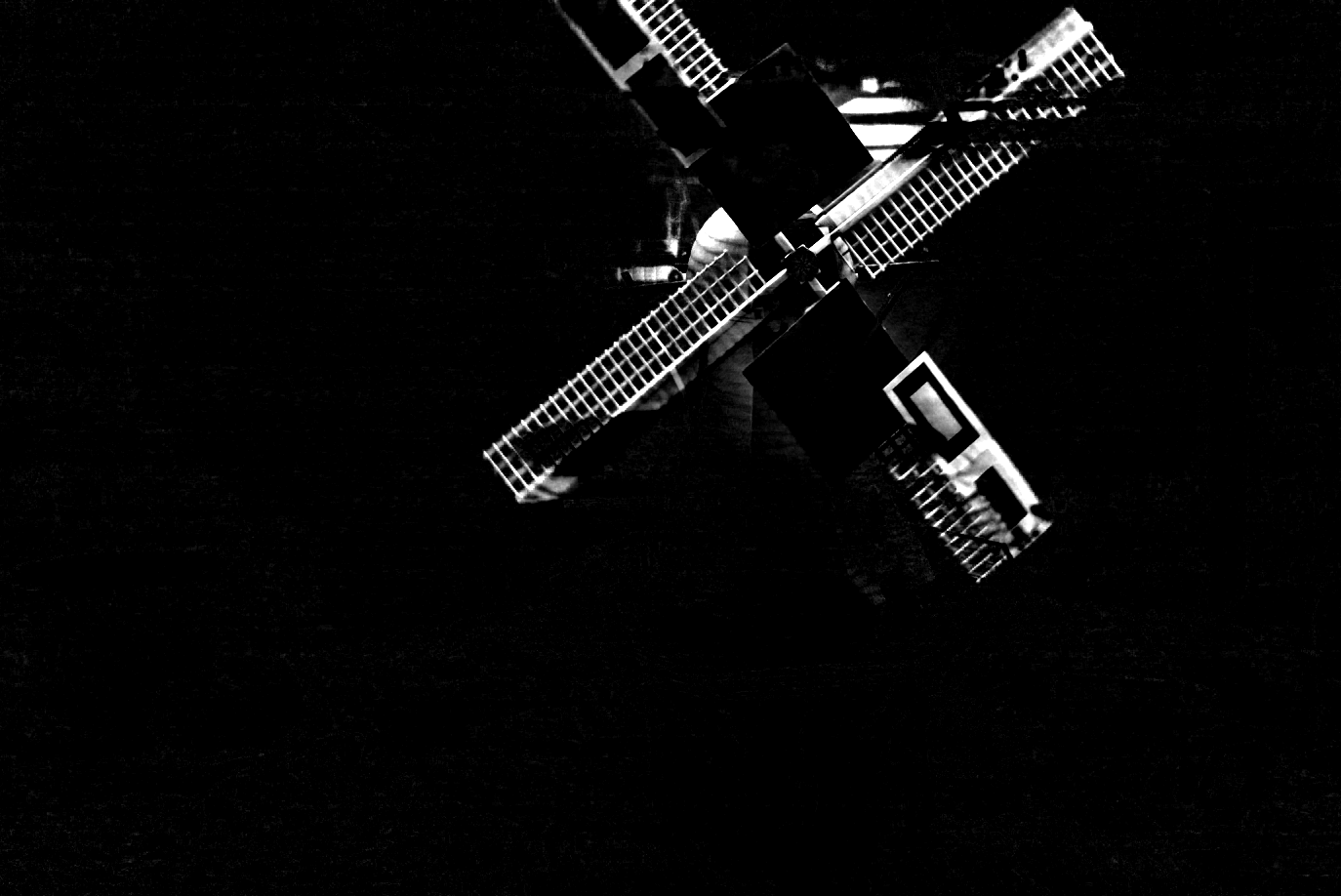}
		\includegraphics[width=0.18\textwidth]{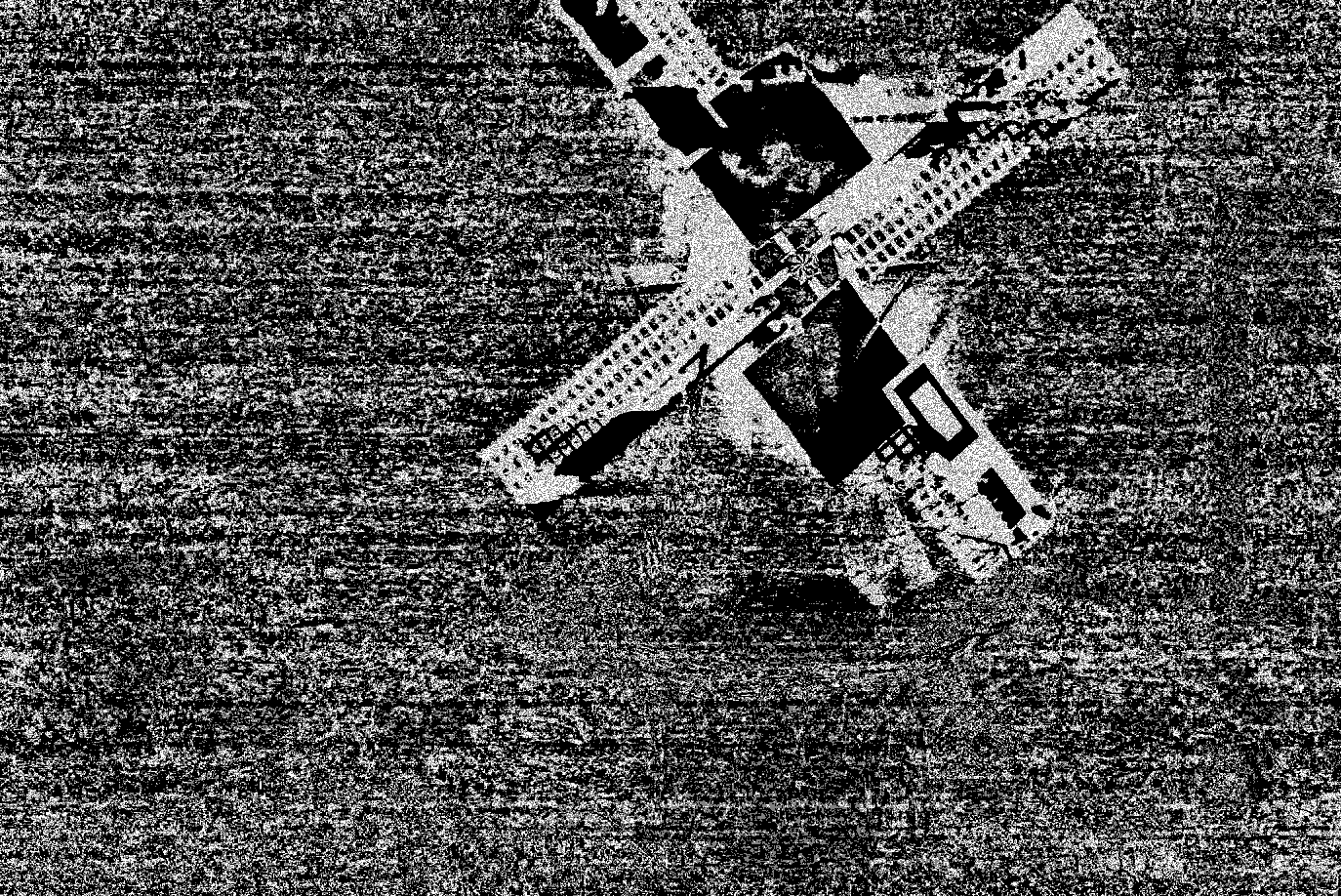}\\
		\textbf{$D$}\hspace{0.16\textwidth}\textbf{$\widehat{L}+\widehat{S}$}\hspace{0.16\textwidth}\textbf{$\widehat{L}$}\hspace{0.16\textwidth}\textbf{$\widehat{S}$}\hspace{0.16\textwidth}\textbf{$\widehat{Z}$}
		\caption{Decomposition of the 30-th, 60-th, 90-th frames for the {\it windmill} video dataset using the approximated solution pairs returned by \hyperref[algo:AltMin]{AltMin}.
			The first row depicts the 30-th raw frames, the second row shows the 60-th raw frames, and the last row shows the 90-th raw frames.
			The first column presents the original raw frames  $D$, the second column shows the denoised image $\widehat{L}+\widehat{S}$, the third and fourth columns display the background $\widehat{L}$ and the moving foreground $\widehat{S}$, respectively, and the last column shows the estimated noise $\widehat{Z}$, calculated as the residual  $D-\widehat{L}-\widehat{S}$.
		}\label{fig:DRV-frame-5}
	\end{figure}
	
	Now, we compare the performances of \hyperref[algo:AltMin]{AltMin} and the ADMM on more real datasets, i.e.,  {\it billiards}, {\it store}, {\it tea set}  and {\it table tennis}.
	Table~\ref{table:realdata_Alt_ADMM} illustrates the comparisons of \hyperref[algo:AltMin]{AltMin} and the ADMM on these datasets.
	Due to the inherent differences in stopping criteria between \hyperref[algo:AltMin]{AltMin} and the ADMM, we present the computational time and the number of iterations required for both algorithms to achieve the same objective function value, specifically defined as the
	maximum of the objective values returned by the two algorithms when the stop tolerance $\epsilon = 10^{-5}$ is met.
	The results demonstrate that \hyperref[algo:AltMin]{AltMin}  outperforms the ADMM on these large-scale dark raw video datasets.
	For {\it store}, {\it tea set} and {\it table tennis} datasets, \hyperref[algo:AltMin]{AltMin}
	exhibits significantly faster convergence rate, resulting in notably fewer iterations and shorter computational time compared to the ADMM.
	
	\begin{table}
		\centering
		\caption{Numerical results of \hyperref[algo:AltMin]{AltMin} and the ADMM on the real dark raw video datasets. The computational time is in the format of ``hours:minutes:seconds". The ``time" and ``iter'' columns represent the computational time and the number of iterations required for both algorithms to achieve the same objective function value,  which is selected as the maximum of the objective values (in the ``obj'' column) returned by \hyperref[algo:AltMin]{AltMin} and the ADMM.}
		\label{table:realdata_Alt_ADMM}
		\begin{tabular}{ccccccc}
			\toprule
			& & & \multicolumn{2}{c}{AltMin} & \multicolumn{2}{c}{ADMM} \\
			\cmidrule(rl){4-5}
			\cmidrule(rl){6-7}
			dataset & $n_2$ & obj & time & iter & time & iter \\
			\midrule
			billiards & 113 & 1726.985 & {\bf 54:02} & {\bf 365} & 58:05 & 612 \\
			store & $112$ & 10007.030 & {\bf 6:02} & {\bf 42} & 31:32 & 326 \\
			tea set & $113$ & { 11729.987} & {\bf 8:01} & {\bf 63} & 13:35 & 204 \\
			table tennis & 114 & 1196.135 & {\bf 2:04} & {\bf 15} & 21:24 & 245\\
			\bottomrule
		\end{tabular}
	\end{table}

	Next, we evaluate the performances of our algorithms \hyperref[algo:AltMin]{AltMin} and \hyperref[algo:Over]{Acc\_AltMin} on datasets {\it basketball player}, {\it windmill}, {\it night street}, and {\it flag}.
	Table~\ref{table:realdata_Alt_Acc} presents the performance comparisons between the two algorithms.
	The results indicate that \hyperref[algo:Over]{Acc\_AltMin} outperforms \hyperref[algo:AltMin]{AltMin} on these datasets, with its successes primarily attributed to the updating strategy of partial SVD (as shown in the ``svd'' columns of Table~\ref{table:realdata_Alt_Acc}).
	For example, for the datasets {\it windmill} and {\it flag}, the time spent by the \hyperref[algo:Over]{Acc\_AltMin} algorithm on SVD operations is only about $60\%$ of that of the \hyperref[algo:AltMin]{AltMin} algorithm.
	Despite the improvements observed on real datasets, the performance gain of \hyperref[algo:Over]{Acc\_AltMin} is not as pronounced as that on the synthetic data in Section~\ref{sec:syn data}.
	We hypothesize that this is due to the ``long and narrow" structure of matrix $D$ (${\rm size}(D) \approx 10^6 \times 10^2$) in these datasets, which makes the updating of $S$ in the algorithms a major computational bottleneck as shown in the ``update $S$'' columns of Table~\ref{table:realdata_Alt_Acc}. This inevitably reduces the relative impact of SVD updates and weakens their primary roles in achieving acceleration.
	
	\begin{table}[!ht]
		\centering
		\caption{Numerical results of algorithm \hyperref[algo:AltMin]{AltMin} and its accelerated counterpart Algorithm \hyperref[algo:Over]{Acc\_AltMin} on the real dark raw video datasets. The third column reports the rank of $\widehat{L}$ returned by the tested algorithms.
			The computational time is in the format of ``hours:minutes:seconds".
			The ``svd'' and ``update $S$'' columns  represent the computational time of SVD(s) operations and the sparse component $S$ update step in  \hyperref[algo:AltMin]{AltMin} and \hyperref[algo:Over]{Acc\_AltMin}, respectively.
		}
		\label{table:realdata_Alt_Acc}
		\begin{tabular}{ccccccccc}
			\toprule
			& & & \multicolumn{3}{c}{AltMin} & \multicolumn{3}{c}{Acc\_AltMin} \\
			\cmidrule(rl){4-6} \cmidrule(rl){7-9}
			dataset & $n_2$ & ${\rm rank}(\widehat{L})$ & time & svd & update $S$ & time & svd & update $S$ \\
			\hline
			basketball player & 112 & 6 & 17:25 & 6:39 & { 8:42} & {\bf 15:23} & {\bf 4:34} & 8:46 \\
			windmill & $112$ & 14 & 50:06 & 20:16 & { 23:30} & {\bf 41:47} & {\bf 11:57} & 23:34 \\
			night street & $112$ & 9 & 26:59 & 9:52 & { 13:43} & {\bf 26:13} & {\bf 9:01} & 13:46 \\
			flag & 111 & 4 & 2:44:38 & 1:10:00 & 1:20:46 & {\bf 2:17:13} & {\bf 42:35} & {1:20:32} \\
			\bottomrule
		\end{tabular}
	\end{table}

	Finally, we illustrate the rank identification property of  \hyperref[algo:AltMin]{AltMin} in Figure \ref{fig:rankidentification}.
	We collect the rank of $L$ in the iterations of \hyperref[algo:AltMin]{AltMin} on eight datasets presented in Tables \ref{table:realdata_Alt_ADMM} and \ref{table:realdata_Alt_Acc}.
	It is evident that \hyperref[algo:AltMin]{AltMin} efficiently identifies the desired rank, with convergence typically occurring within $30$ iterations.
	However, in some cases, the algorithm may require several dozen of iterations to achieve an approximate solution with the desired accuracy.
	
	\begin{figure}[!htbp]
		\centering
		\includegraphics[width=0.85\textwidth]{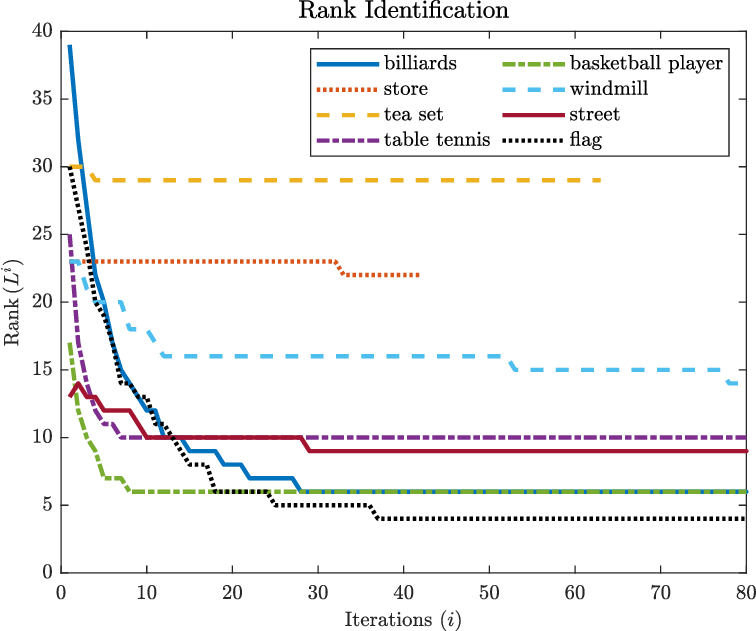}
		\caption{The rank identification of Algorithm  \hyperref[algo:AltMin]{AltMin}.}
		\label{fig:rankidentification}
	\end{figure}

	\section{Conclusion}
	
	In this paper, we study the robust matrix recovery SRPCP model from an optimization perspective.
	We establish the equivalent DRO reformulation of this model, providing a new perspective for understanding the SRPCP model.
	More important, we develop an efficient tuning-free alternating minimization algorithm, i.e., \hyperref[algo:AltMin]{AltMin}, to solve the SRPCP problem.
	Each iteration of this algorithm involves subproblems with closed-form optimal solutions.
	Furthermore, we accelerate the \hyperref[algo:AltMin]{AltMin} by using the variational formulation of the nuclear norm and the Burer-Monteiro decomposition.
	This accelerated algorithm, i.e., \hyperref[algo:Over]{Acc\_AltMin} also features closed-form optimal solutions for the subproblems in each iteration.
	Numerical experiments on various large-scale datasets demonstrate the efficiency and robustness of our proposed algorithms.
	In all tests, \hyperref[algo:AltMin]{AltMin} exhibits superior performance, while \hyperref[algo:Over]{Acc\_AltMin} achieves significant improvements over \hyperref[algo:AltMin]{AltMin} in low-rank scenarios.

	\bibliographystyle{abbrv}
	\bibliography{main.bib}

	\appendix
		\section{Proof of Theorem \ref{thm:dro_formulation}}\label{app:dro_formulation}
		\begin{proof}
		Denote $E_{jk}:=[e_j(n_1)e_k^T(n_2);e_j(n_1)e_k^T(n_2)]$ and $m:=|\Omega_{\rm obs}|$. Let the empirical distribution be given by $\mathbb{P}_{m} = \frac{1}{m} \sum_{(j,k)\in \Omega_{\rm obs}} {\bf 1}_{\{(E_{jk},D_{jk})\}}$. Then the right-hand side of \eqref{eq:DRO} equals to
		$$
		\delta \cdot \inf_{L,S} \, \left\{ \frac{1}{\sqrt{m\delta}} \sqrt{\sum_{(i,j)\in \Omega_{\rm obs}} (L_{ij} + S_{ij} - D_{ij} )^2} + \|L\|_* + \lambda\|S\|_1 \right\}^2.
		$$
		Meanwhile, from \eqref{eq:loss}, the left-hand side of \eqref{eq:DRO} equals to
		\begin{equation*}
			\begin{array}{cl}
				\displaystyle\inf_{L,S}\sup_{\xi_{jk},\zeta_{jk},(j,k)\in\Omega_{\rm obs}} & \displaystyle \frac{1}{m} \sum_{(j,k)\in\Omega_{\rm obs}}{\rm E}\Big[ \big((D - L - S)_{jk} + \langle \xi_{jk},L\rangle + \langle \zeta_{jk},S\rangle\big)^2\Big]\\[10pt]
				{\rm s.t.} & \displaystyle
				\frac{1}{m} \sum_{(j,k)\in\Omega_{\rm obs}}{\rm E} \Big[\max\big\{\| \xi_{jk}\|^2,\| \zeta_{jk}\|_{\infty}^2/\lambda^2\big\}\Big] \leq \delta.
			\end{array}
		\end{equation*}
		The proof is thus completed by letting $\delta = 1/(\mu^2m)$.
		\end{proof}
		
		\section{Proof of Proposition \ref{prop:1}}\label{app:prop_1}
		
		\begin{proof}
		For statement (i), suppose that $a_i > 0$. If $\bar{s}_i > a_i$, we construct a vector $\hat{s}$ by $\hat{s}_j = \bar{s}_j \mbox{ if } j\neq i; \hat{s}_j = a_i \mbox{ if } j = i$; if $\bar{s}_i < 0$, we construct the vector $\hat{s}$ by $\hat{s}_j = \bar{s}_j \mbox{ if } j\neq i; \hat{s}_j = 0 \mbox{ if } j = i$. In both cases, it is not difficult to verify that $\hat{s}$ achieves a lower objective function value than $\bar{s}$. 
		
		Statements (ii) and (iii) can be proved by using similar arguments. Statement (iv) holds due to (i), (ii), and (iii). Statement (v) is true since $\|s-Pa\|_2 + \tau \|s\|_1 = \|P^Ts-a\|_2 + \tau \|P^Ts\|_1$. 
		
		\end{proof}
		
		\section{Proof of Proposition \ref{pro:unique}}\label{app:unique}
		\begin{proof}
		We shall prove (i) by discussing two cases of $\tau$, i.e., case 1 with $\tau > {1}/{\|a\|_0}$ and case 2 with $\tau < {1}/{\|a\|_0}$.
		
		In case 1 with $\tau > {1}/{\|a\|_0}$, let $\bar{s}$ and $\hat{s}$ be two optimal solutions to \eqref{op:1}. From (i) in Proposition~\ref{prop:1}, we have that $0 \leq \bar{s} \leq a$ and $0 \leq \hat{s} \leq a$. Denote the optimal function value of \eqref{op:1} by $h_*:=h(\bar{s}) = h(\hat{s})$. For any $t\in(0,1)$, we know that $s^t:= t\bar{s} + (1-t)\hat{s}$ is optimal to \eqref{op:1}. Therefore, we have that
		$$
		\begin{array}{l}
			0 = h(s^t) - th(\bar{s}) - (1-t)h(\hat{s})\\
			= \|t\bar{s} + (1-t)\hat{s}-a\|_2 + \tau \|t\bar{s} + (1-t)\hat{s}\|_1 - t(\|\bar{s} - a\|_2 + \tau\|\bar{s}\|_1) - (1-t)(\|\hat{s} - a\|_2 + \tau\|\hat{s}\|_1)\\
			= \|t(\bar{s} - a) + (1-t)(\hat{s}-a)\|_2 - t\|\bar{s} - a\|_2 - (1-t)\|\hat{s} - a\|_2,
		\end{array}
		$$
		where the last equality is due to the fact that $\|t\bar{s} + (1-t)\hat{s}\|_1 = t\|\bar{s}\|_1 + (1-t)\|\hat{s}\|_1$. This further implies that $\bar{s} - a$ and $\hat{s} - a$ are parallel. By switching $\bar{s}$ and $\hat{s}$ if necessary, we can find $u\in [0,1]$ such that $ \bar{s} - a = u(\hat{s} - a)$.  Thus, we have that
		$$
		h(\bar{s}) = h(u\hat{s} + (1-u)a) = u\|\hat{s}-a\|_2 + \tau u\|\hat{s}\|_1 + \tau (1-u)\|a\|_1 = uh(\hat{s}) + (1-u)h(a),
		$$
		which implies that $(1-u)(h_* - h(a))=0$.
		From (ii) in Theorem \ref{thm:1}, we know that when $\tau > 1/\|a\|_0$, the vector $a$ is not optimal. Thus, $h_* < h(a)$ and we deduce that $u=1$, i.e., $\bar{s}=\hat{s}$.
		
		When $\tau < {1}/{\|a\|_0}$ in case 2, by (ii) in Theorem \ref{thm:1}, the vector $a$ is optimal. Let $\hat{s}$ be an optimal solution to \eqref{op:1}.  Again, we know from (i) in Proposition \ref{prop:1} that $0 \leq \hat{s} \leq a$.  Therefore, it holds that \[0 = h(\hat{s}) - h(a) =  \sqrt{\sum_{\{i\mid a_i>0\}}(a_i - \hat{s}_i)^2} - \tau \sum_{\{i\mid a_i>0\}}(a_i - \hat{s}_i) \geq (\frac{1}{\sqrt{\|a\|_0}} - \tau)\sum_{\{i\mid a_i>0\}}(a_i - \hat{s}_i) \ge 0.\]
		This implies that $\hat{s} = a$.
		We thus prove that in both cases, the optimal solution set is a singleton.
		
		For (ii), we know from (ii) in Theorem \ref{thm:1} that $a$ is optimal. The optimality of $a^{\varepsilon}$ can be easily verified by noting that $h(a^{\varepsilon}) = h(a)$.
		\end{proof}
		
		\section{Proof of Proposition \ref{pro:diagA}}\label{app:diagA}
		\begin{proof}
		Let $r:=\|\sigma\|_0$ and $\overline{L}:=\begin{pmatrix}
			{\rm Diag}(d_{\rho}(\sigma))  \\
			0
		\end{pmatrix}$. We prove the result by examining the optimality of $\overline{L}$ under different ranges of $\rho$.
		Suppose $\rho \geq {\|\sigma\|_{\infty}}/{\|\sigma\|_2}$. Then, $\left\| \frac{A}{\rho\|A\|_F} \right\| = \frac{\|\sigma\|_{\infty}}{\rho\|\sigma\|_2} \leq 1.$
		The optimality of $\overline{L}=0$ can be verified by checking the following inclusion system $0\in-\frac{A}{\|A\|_F} + \rho \{ W \in\mathbb{R}^{n_1\times n_2}\,|\,\|W\| \leq 1\}.$
		Suppose that $\rho \leq {1}/{\sqrt{r}}$.  The subgradient of the nuclear norm at $A$ is given by
		$ \partial \|A\|_*= \left\{
		\begin{pmatrix}
			I_r\\0
		\end{pmatrix}
		\right\}$
		if $A$ is of full rank ($r=n_2$), and
		$\partial \|A\|_*=\left\{
		\begin{pmatrix}
			I_r & \\
			& W
		\end{pmatrix}\,\Bigg|\,
		W \in\mathbb{R}^{(n_1-r)\times (n_2-r)},\|W\| \leq 1
		\right\}
		$  otherwise.
		Now, we can verify the  optimality of  $\overline{L}=A$ by noting the fact that
		$0\in \left\{ W \in\mathbb{R}^{n_1\times n_2}\,\middle|\,\|W\|_F \leq 1\right\} + \rho \partial \|A\|_*.$
		
		Lastly, suppose  that ${1}/{\sqrt{r}} < \rho < {\|\sigma\|_{\infty}}/{\|\sigma\|_2}$. In the definition of $d_\rho$, we know from (iii) in Theorem~\ref{thm:1} that $d_{\rho}(\sigma) = [\sigma_1- t_k;\dots;\sigma_k- t_k;0;\dots;0]$, where $1\le k\le r-1$ and $t_k=\rho\|d_{\rho}(\sigma) - \sigma\|_2$ are given in \eqref{search:k2}. Then, the optimality condition  of \eqref{op:L} at $\overline{L}$ holds since we have that
		$$
		-\frac{1}{\rho} \partial \|\overline{L} - A\|_F \ni -\frac{1}{\rho}\frac{\overline{L} - A}{\|\overline{L} - A\|_F} =
		{\rm Diag} \left(
		I_k,\frac{\sigma_{k+1}}{t_k},\dots, \frac{\sigma_r}{t_k}, 0_{(n_1-r)\times (n_2-r)} \right)
		\in \partial \|\overline{L}\|_*,
		$$
		where the last inclusion is due to the fact that $\sigma_i \leq t_k$ for $r \ge i \geq k+1$. The proof is completed.
		\end{proof}

		\section{Proof of Theorem \ref{thm:convergence}}\label{app:convergence}
		\begin{proof}
		Since $f$ is coercive, we know that the level set $X^0=\{(L,S)\,|\,f(L,S) \leq f(L^0,S^0)\}$ is compact. Moreover, the updating rules in Algorithm  \hyperref[algo:AltMin]{AltMin} imply that $f(L^{i+1},S^{i+1}) \leq f(L^{i+1},S^{i}) \leq f(L^{i},S^{i}) $, and therefore $\{(L^{i},S^{i})\}_{i\geq 0} \subseteq  X^0 $ is a bounded sequence. Let $\{(L^{i},S^{i})\}_{i\in\mathcal{D}}$ be any convergent subsequence that converges to $(\overline{L},\overline{S})$, where $\mathcal{D}\subseteq \{0,1,\dots\}$ is an index set. We can assume that  $\{L^{i+1}\}_{i\in\mathcal{D}}$ converges to a limit $\widehat{L}$, by taking a subsequence if necessary.
		By the monotonicity of function values,  we obtain that
		\begin{equation}\label{formula1}
			\lim_{i\to \infty} f(L^{i},S^{i}) = \lim_{i\to \infty} f(L^{i},S^{i+1}) \mbox{ and  }    f(\overline{L},\overline{S}) = f(\widehat{L},\overline{S}).
		\end{equation}
		The updating formulas \eqref{update-L} and \eqref{update-S} yield that for any $i\geq 0$,
		$
		f(L^{i+1},S^{i}) \leq f(L^{i+1} + d_L,S^{i}), \ \forall\,d_L,$ and $
		f(L^{i+1},S^{i+1}) \leq f(L^{i+1},S^{i+1} + d_S),  \ \forall\,d_S.
		$
		Since $f$ is continuous, we take the limit and obtain that
		\begin{align}
			f(\widehat{L},\overline{S}) \leq f(\widehat{L} + d_L,\overline{S}), \quad \forall\,d_L,\quad
			f(\overline{L},\overline{S}) \leq f(\overline{L},\overline{S} + d_S), \quad \forall\,d_S. \label{formula3}
		\end{align}
		Then, \eqref{formula1} and \eqref{formula3} imply that
		$
		f(\overline{L},\overline{S}) \leq f(\overline{L} + d_L,\overline{S}), \ \forall\,d_L.
		$
		This, together with \eqref{formula3}, implies that $(\overline{L},\overline{S})$ is a coordinatewise minimum point of $f$.
		In addition, if $f$ is non-overfitting at $(\overline{L},\overline{S})$, then by Proposition~\ref{prop:9}, $f$ is regular at $(\overline{L},\overline{S})$. Therefore, $(\overline{L},\overline{S})$ is also a stationary point of $f$. By the convexity of $f$, this is in fact a minimum point of $f$. The proof is completed.
		\end{proof}
		
		\section{Proof of Theorem \ref{thm:overparameter}}\label{sec:proof_thm_7}

		\begin{proof}
		First, we consider the simple case with $c=0$, problem \eqref{op:uv} reduces to problem \eqref{op:2}, which has already been studied in Section~ \ref{subsec:updateS}.
		Hence, we focus specifically on the case where $c>0$.
		Using similar arguments to (iii) in Proposition \ref{prop:1}, we can assume without loss of generality that $w_k>0$, and this implies $\ell=\|w\|_0=k$.
		To simplify the notation, we denote $d_\rho(w,c)$ as $d$.
		We will prove Theorem \ref{thm:overparameter} in the following three cases:
		$ {\rm (i)} \frac{\|w\|_\infty}{\sqrt{\|w\|_2^2 +c^2}}\leq \rho, \
		{\rm (ii)} 0< \rho\leq \frac{w_k}{\sqrt{k w_k^2+c^2}}, \
		{\rm (iii)} \frac{w_k}{\sqrt{k w_k^2+c^2}}< \rho<\frac{\|w\|_\infty}{\sqrt{\|w\|_2^2+c^2}}.
		$
		
		We begin with the case (i) with ${\|w\|_\infty}/{\sqrt{\|w\|_2^2 +c^2}}\leq \rho$.
		By simple calculations, we have that
		$
		\rho - \frac{\|w\|_1}{\sqrt{\|w\|_2^2+c^2}} \geq 0\iff
		\rho 1_k - \frac{w}{\sqrt{\|w\|_2^2+c^2}} \in \mathbb{R}^k_+
		\iff
		0\in -\frac{w}{\sqrt{\|w\|_2^2+c^2}} + \rho 1_k + \mathbb{R}^k_-
		\overset{\eqref{normal-cone}}{\iff}
		0\in -\frac{w}{\sqrt{\|w\|_2^2+c^2}} + \rho 1_k + \mathcal{N}_{\mathbb{R}^k_+}(0).
		$
		The last equation shows that the optimality condition of \eqref{op:uv} holds at $d=0$. Thus, $0$ is a solution to \eqref{op:uv}.
		Then, we turn to the case (ii) with $0< \rho\leq {w_k}/{\sqrt{k w_k^2+c^2}}$. In this range, we have that $k\rho^2<1$.
		First, we prove that $w_k - \rho\sqrt{\frac{c^2}{1-k\rho^2} }\geq0$ by noting that %
		\begin{equation*}
			w_k - \rho\sqrt{\frac{c^2}{1-k\rho^2 }}\geq0\iff w_k^2 \geq \frac{\rho^2c^2}{1-k\rho^2}\iff \rho^2\leq\frac{w_k^2}{k w_k^2 + c^2}.
		\end{equation*}
		Therefore, we have that $w_i-\rho\sqrt{\frac{c^2}{1-k\rho^2 }}\geq 0$ for all $i\in\{1,\dots,k\}$ since $w$ is sorted in descending order.
		Then, we  verify that the following optimality condition of problem \eqref{op:uv} holds at $d=w-\rho\sqrt{\frac{c^2}{1-k\rho^2 }}1_k$:
		\begin{equation*}
			\frac{d-w}{\sqrt{\|d-w\|_2^2+c^2}} + \rho 1_k + \mathcal{N}_{\mathbb{R}_{+}^k}(d) = -\frac{\rho\sqrt{\frac{c^2}{1-k\rho^2 }}\, 1_k}{\sqrt{\frac{k\rho^2 c^2}{1-k\rho^2}+c^2}} + \rho 1_k + \mathbb{R}^k_-\ni0,
		\end{equation*}
		where the first equation  follows from %
		$d\geq0$ and the definition of the normal cone \eqref{normal-cone}.
		
		Lastly, consider the case (iii) with ${w_k}/{\sqrt{k w_k^2+c^2}} <\rho<{\|w\|_\infty}/{\sqrt{\|w\|_2^2+c^2}}$.
		We define the sequence
		$$
		s_j = \frac{w_{j}}{\sqrt{(j-1)w_{j}^2+w_{j}^2+\cdots+w_{k}^2+c^2}},\ j=1,2,\dots,k.
		$$
		By simple computations, we know $\{s_j\}$ is a nonincreasing sequence satisfying that
		$$\frac{w_k}{\sqrt{k w_k^2+c^2}}=
		s_k \leq s_{k-1} \leq \cdots \leq s_2\leq s_1=\frac{w_1}{\sqrt{\|w\|_2^2+c^2}}.
		$$
		Therefore, there exists $\bar{\ell}\in\{1,2,\dots,k-1\}$ such that
		$\rho\in [ s_{\bar{\ell}+1}, s_{\bar{\ell}})$.
		Now we prove the existence of $i\in\{1,2,\dots,\ell-1\}$ satisfying the condition \eqref{search:rk2}.
		On one hand, $\rho < s_{\bar{\ell}}$ implies that $\bar{\ell}<\bar{\ell}-1+\frac{w_{\bar{\ell}}^2+\cdots+w_{k}^2+c^2}{w_{\bar{\ell}}^2}\leq\frac{1}{\rho^2}$.
		On the other hand, 	$\rho \geq  s_{\bar{\ell}+1}$ implies that $w_{\bar{\ell}+1}^2\leq \frac{w_{\bar{\ell}+1}^2+\dots+w_k^2+c^2}{\frac{1}{\rho^2}-\bar{\ell}}$. Thus, we further have that
		\begin{align*}
			\big[ w_{j+1}-t_j\big]\big|_{j=0}
			=w_1-\rho\sqrt{\|w\|_2^2+c^2}>0, \quad
			\big[ w_{j+1}-t_j\big]\big|_{j=\bar{\ell}}
			=w_{\bar{\ell}+1} - \sqrt{ \frac{w_{\bar{\ell}+1}^2+\dots+w_k^2+c^2}{\frac{1}{\rho^2}-\bar{\ell}}}
			\leq 0.
		\end{align*}
		Consequently, there exists $i\in\{1,\dots,\bar{\ell}\}$ such that
		$
		w_i-t_{i-1}>0 \mbox{ and } w_{i+1}-t_i\leq 0.
		$
		Then, we show that $w_i>t_i$. In fact, since $w_i>t_{i-1}$, we obtain that
		\begin{equation}\label{eq:wt-1}
			\begin{aligned}
				&w_i^2>\frac{w_i^2+w_{i+1}^2+\dots+w_k^2+c^2}{\frac{1}{\rho^2}-i+1}
				\iff\quad (\frac{1}{\rho^2}-i+1)w_i^2>w_i^2+w_{i+1}^2+\dots+w_k^2+c^2\\
				\iff\quad&(\frac{1}{\rho^2}-i)w_i^2>w_{i+1}^2+\dots+w_k^2+c^2
				\iff\quad w_i^2>\frac{w_{i+1}^2+\dots+w_k^2+c^2}{\frac{1}{\rho^2}-i}=t_i^2.
			\end{aligned}
		\end{equation}
		Thus, the existence of $i$ in \eqref{search:rk2} is established.
		Next, we show the uniqueness of $i$. Suppose that there exist two distinct integers $i$ and $l$ such that $1\leq i< l\leq \bar{\ell}$, $w_i-t_{i-1}>0$, $w_{i+1}-t_i\leq 0$, $w_l-t_{l-1}>0$, and $w_{l+1}-t_l\leq 0$.
		Namely, we have that
		$
		\big[ w_{j+1}-t_j\big]\big|_{j=i}\leq0,\ \big[ w_{j+1}-t_j\big]\big|_{j=l-1}>0,
		$
		which implies that there exists $p\in\{i+1,\dots,l-1\}$ such that $w_p-t_{p-1}\leq 0$ and $w_{p+1}-t_p>0$.
		Similar to \eqref{eq:wt-1}, from $w_p\leq t_{p-1}$, we obtain $w_p\leq t_p$. This implies $w_p\leq t_p<w_{p+1}$, which contradicts the sorted assumption \eqref{assum:uv}. Therefore, we have proved the uniqueness of $i$.
		Finally, through simple calculations, we can verify that the optimality condition of \eqref{op:uv} holds at $d$ in \eqref{eq:lowrankd}:
		\begin{equation*}%
			0\in\frac{d-w}{\sqrt{\|d-w\|^2+c^2}}+\rho 1_k+\mathcal{N}_{\mathbb{R}_+^k}(d)\overset{\eqref{normal-cone}}{=}\frac{d-w}{\sqrt{\|d-w\|^2+c^2}}+\rho 1_k+[0_i;\mathbb{R}_-^{k-i}].
		\end{equation*}
		This completes the proof for the theorem.		
		\end{proof}

		\section{Proof of Proposition \ref{prop:condition_k_geq_rankL}}\label{sec:proof_prop_8}
		\begin{proof}
		To prove this conclusion, we will divide the proof into two parts:
		\begin{enumerate}
			\item[Claim 1.] When \eqref{eq:condition_k_geq_rankL} is satisfied, we show that $[d_\rho(w,c);0_{n_2-k}]$ is an optimal solution to \eqref{op:2}.
			\item[Claim 2.] When \eqref{eq:condition_k_geq_rankL} is not satisfied, i.e., $0<\rho<{a_{\ell+1}}/{\sqrt{\ell  a_{\ell+1}^2+c^2}}$, we argue that $[d_\rho(w,c);0_{n_2-k}]$ is not an optimal solution to \eqref{op:2}.
		\end{enumerate}
		For the first part, by noting \eqref{eq:lowrankd} and
		$\frac{a_{\ell+1}}{\sqrt{\ell a_{\ell+1}^2+c^2}}\leq \frac{a_{\ell}}{\sqrt{\ell a_{\ell}^2+c^2}}=\frac{w_{\ell}}{\sqrt{\ell w_{\ell}^2+c^2}}\leq \frac{\|w\|_\infty}{\sqrt{\|w\|_2^2+c^2}},$
		we will prove Claim 1 in the following three cases:
		$
		{\rm (i)}\frac{\|w\|_\infty}{\sqrt{\|w\|_2^2 +c^2}}\leq\rho, \
		{\rm (ii)} \frac{w_\ell}{\sqrt{\ell w_\ell^2+c^2}} <\rho
		<\frac{\|w\|_\infty}{\sqrt{\|w\|_2^2+c^2}}, \
		{\rm (iii)} \frac{a_{\ell+1}}{\sqrt{\ell  a_{\ell+1}^2+c^2}}\leq\rho\leq\frac{w_\ell}{\sqrt{\ell  w_\ell^2+c^2}}.
		$
		We begin with the case (i) with ${\|w\|_\infty}/{\sqrt{\|w\|_2^2 +c^2}}\leq\rho$. From Theorem \ref{thm:overparameter}, we know that $d_\rho(w,c)=0_k$.
		Then, $[d_\rho(w,c);0_{n_2-k}]=0_{n_2}$ is an optimal solution to \eqref{op:2} due to Theorem \ref{thm:1}.
		Now, for case (ii) with  ${w_\ell}/{\sqrt{\ell w_\ell^2+c^2}} <\rho
		<{\|w\|_\infty}/{\sqrt{\|w\|_2^2+c^2}}$, we know from Theorem \ref{thm:overparameter} that there exists $i\in\{1,2,\dots,\ell-1\}$ such that $t_i$ satisfies \eqref{search:rk2}.
		Since
		\begin{equation*}
			\rho>\frac{w_\ell}{\sqrt{\ell w_\ell^2+c^2}}\geq\frac{1}{\sqrt{\|a\|_0}} \mbox{ and }
			\frac{\|w\|_\infty}{\sqrt{\|w\|_2^2+c^2}} = \frac{\|a\|_\infty}{\|a\|_2},
		\end{equation*}
		we know that $\rho\in({1}/{\sqrt{\|a\|_0}},{\|a\|_\infty}/{\|a\|_2})$ in this case.
		Observing that $t_i$ now also satisfies condition \eqref{search:k2}, hence, it holds from Theorem (iii) in \ref{thm:1} that $[d_\rho(w,c);0_{n_2-k}]$ is an optimal solution to \eqref{op:2}.
		
		Now, we consider the case (iii) with  ${a_{\ell+1}}/{\sqrt{\ell  a_{\ell+1}^2+c^2}}\leq\rho\leq{w_\ell}/{\sqrt{\ell  w_\ell^2+c^2}}$.
		If $c = 0$,
		then $a_{k+1}=\dots=a_{n_2}=0$, and
		\[ d_\rho(w,c)=w \mbox{ and }
		\frac{w_\ell}{\sqrt{\ell  w_\ell^2+c^2}}=\frac{1}{\sqrt{\|a\|_0}}.
		\]
		Then, according to (ii) in Theorem \ref{thm:1}, $[d_\rho(w,c);0_{n_2-k}]=[w;0_{n_2-k}]=a$ is an optimal solution  to \eqref{op:2}.
		Lastly, if $c>0$, then $\ell=\|w\|_0 = k$.
		Since
		\begin{equation}\label{eq:mu left range}
			\rho^2\geq \frac{a_{k+1}^2}{c^2+ka_{k+1}^2}\geq\frac{a_{k+1}^2}{(\|a\|_0-k)a_{k+1}^2+ka_{k+1}^2}\geq \frac{1}{\|a\|_0},
		\end{equation}
		we divide the following proof into two cases $\rho^2={1}/{\|a\|_0}$ and $\rho^2>{1}/{\|a\|_0}$.
		When $\rho^2={1}/{\|a\|_0}$, we have from \eqref{eq:mu left range} that
		$c^2 = (\|a\|_0 - k)a_{k+1}^2 = \sum_{i=k+1}^{n_2} a_i^2$, and thus $a_{k+1} = \min\{a_i|a_i>0\}$.
		Then,  (i) in Theorem \ref{thm:1} and (ii) in Proposition \ref{pro:unique}  imply that
		\[[d_\rho(w,c);0_{n_2-k}]=[w-a_{k+1};0_{n_2-k}]=a - a_{k+1}{\rm sign}(a) \odot 1_{n_2}\] is an optimal solution to \eqref{op:2}.
		Now, we consider the case with $\rho^2>{1}/{\|a\|_0}$. From Remark \ref{remark:ineq}, we know that
		\begin{equation}\label{eq:mu right range}
			\rho\leq \frac{w_k}{\sqrt{k w_k^2+c^2}}\leq\frac{\|w\|_\infty}{\sqrt{\|w \|_2^2 + c^2}}=\frac{\|a\|_\infty}{\|a\|_2}.
		\end{equation}
		If $\rho\in\left({1}/{\sqrt{\|a\|_0}},{\|a\|_\infty}/{\|a\|_2}\right)$,
		then
		\[ t_k:=\sqrt{\frac{a_{k+1}+\dots+a_{n_2}^2}{\frac{1}{\rho^2}-k}}=\rho\sqrt{\frac{c^2}{1-k\rho^2}}<w_k,\]
		where the last inequality follows due to  $\rho< \|a\|_\infty/ \|a\|_2 = {w_{k}}/{\sqrt{\|w\|^2+c^2}}$.
		This, together with (iii) in Theorem \ref{thm:1} and the first inequality in \eqref{eq:mu right range} that $a_{k+1}\leq \rho\sqrt{{c^2}/(1-k\rho^2)}=t_k$, implies that  $[d_\rho(w,c);0_{n_2-k}]$ is an optimal to \eqref{op:2}.
		Lastly, if $\rho = {\|a\|_\infty}/{\|a\|_2}$, \eqref{eq:mu right range} implies $w_1=\dots=w_k>0$.
		Thus,
		\[
		\rho\sqrt{\frac{c^2}{1-k\rho^2}} = \frac{w_k}{\|a\|_2} \frac{c}{\sqrt{1 - \frac{k w_k^2}{\|a\|_2^2}}} =
		\frac{w_k}{\|a\|_2} \|a\|_2 = w_k.
		\]
		Hence, we have from \eqref{eq:lowrankd} in Theorem \ref{thm:overparameter} that
		$d_\rho(w,c)=w-w_k 1_k = 0_k$.
		Then, (i) in Theorem \ref{thm:1}  asserts that
		$[d_\rho(w,c);0_{n_2-k}]=0_{n_2}$ is an optimal solution to \eqref{op:2}.
		We thus complete the proof for Claim 1.
		
		Now, we turn to prove Claim 2, i.e., if $0<\rho<{a_{\ell+1}}/{\sqrt{\ell a_{\ell+1}^2+c^2}}$, then $[d_\rho(w,c);0_{n_2-k}]$ cannot be an optimal solution to \eqref{op:2}.
		Since $0<{a_{\ell+1}}/{\sqrt{\ell a_{\ell+1}^2+c^2}}$, it holds that $\ell = k$ and $a_{k+1} = a_{l+1} > 0$. Otherwise, if $\ell < k$, then $a_{\ell+1} =0$, a contradiction. Then,
		\begin{equation}\label{eq:c_and_rho}
			c = \sqrt{\sum_{i=k+1}^{n_2} a_i^2} \ge a_{k+1} > 0 \mbox{ and }
			0 < \rho < \frac{a_{k+1}}{\sqrt{k a_{k+1}^2+c^2}} < \frac{1}{\sqrt{k}}.
		\end{equation}
		Thus, we know that $d_\rho(w,c)\neq w$ from \eqref{eq:lowrankd}.
		Now, we verify the optimality condition of \eqref{op:2} at $\tilde{s}:=[d_\rho(w,c);0_{n_2-k}]$, i.e., whether $0_{n_2}$ belongs to the subgradient
		\begin{equation*}
			{\cal M}_{\tilde s}:= \frac{\tilde s-a}{\|\tilde s-a\|_2}+\rho 1_{n_2} + \mathcal{N}_{\mathbb{R}^{n_2}_+}(\tilde s).
		\end{equation*}
		For any $v\in {{\cal M}_{\tilde s}}$, by noting the definition of $\mathcal{N}_{\mathbb{R}^{n_2}_+}(\tilde s)$ and $\tilde s_{k+1} = 0$, we see that
		\begin{equation}\label{eq:v_k1}
			v_{k+1} \le \frac{ -a_{k+1}}{\|\tilde s-a\|_2}+ \rho = -\frac{a_{k+1}}{\sqrt{\frac{c^2}{1-k\rho^2}}} +\rho < 0,
		\end{equation}
		where the equality follows from \eqref{eq:lowrankd} and the second inequality holds by \eqref{eq:c_and_rho} and simple calculations.
		Now, \eqref{eq:v_k1} asserts that
		$0\not\in {\cal M}_{\tilde s}$, and thus $[d_\rho(w,c);0_{n_2-k}]$ is not an optimal solution to \eqref{op:2}. This ends the proof for the theorem.		
		\end{proof}
		
\end{document}